\documentclass[12pt]{article}

\usepackage{fullpage}
\usepackage{amsmath}
\usepackage{amssymb}
\usepackage{amsthm}
\usepackage{graphicx}
\usepackage{placeins}
\usepackage{extpfeil}
\usepackage{enumerate}
\usepackage{caption}
\usepackage{subcaption}
\usepackage{hyperref}
\usepackage{xcolor}
\usepackage{pbox}
\usepackage[numbers,sort&compress]{natbib}

\usepackage{multirow}

\newcommand{\be}{\begin{equation}} 
\newcommand{\ee}{\end{equation}}
\newcommand{\bee}{\begin{equation*}}
\newcommand{\eee}{\end{equation*}}

\newtheorem{thm}{Theorem}[section]
\newtheorem{theorem}{Theorem} 

\newtheorem{prop}{Proposition}[section]

\newtheorem{rmk}{Remark}[section]
\newtheorem{lem}{Lemma}[section]

\newtheorem{cor}{Corollary}[section]
\newtheorem{example}{Example}[section]
\newtheorem{que}{Question}[section]

\newcommand{\ur}{u_{r}}
\newcommand{\uthe}{u_{\theta}}

\newcommand{\uphi}{u_{\phi}}

\newcommand{\dive}{{\rm div \hspace{0.05cm} }}

\newcommand{\ctthe}{\cot\theta}
\newcommand{\sthe}{\sin\theta}
\newcommand{\cthe}{\cos\theta}

\newcommand{\R}{\mathbb{R}}
\newcommand{\bS}{\mathbb{S}}
\newcommand{\bC}{\mathbb{C}}

\newcommand{\bmx}{\begin{bmatrix}}
\newcommand{\emx}{\end{bmatrix}}

\begin{document}

\title{Recent research on $(-1)$-homogeneous solutions of stationary Navier-Stokes equations}

\author{
	Li Li\footnote{School of Mathematics and Statistics, Ningbo University, 818 Fenghua Road, Ningbo, Zhejiang 315211, China. Email: lili2@nbu.edu.cn. Partially supported by NSFC grant 12271276 and ZJNSF grant LR24A010001.}, 
	Xukai Yan\footnote{Department of Mathematics, Oklahoma State University, 401 Mathematical Sciences Building, Stillwater, OK 74078, USA. Email: xuyan@okstate.edu. Partially supported by NSF Career Award DMS-2441137 and Simons Foundation Travel Support for Mathematicians 962527.}
	}
\date{}
\maketitle

\abstract{
  We make an exposition of recent research on $(-1)$-homogeneous solutions of the three-dimensional incompressible stationary Navier-Stokes equations with singular rays. We also discuss properties of such solutions that are axisymmetric with no swirl, and present graphs illustrating examples that exhibit various typical types of singular behavior.
}

\section{Introduction}\label{sec:intro}

We study the three-dimensional incompressible stationary Navier-Stokes equations, 
\begin{equation}\label{NS}
	\left\{
	\begin{aligned}
		& -\nu\Delta u + ( u \cdot \nabla ) u + \nabla p = 0, \\ 
		& \dive u=0,
	\end{aligned}
	\right.
\end{equation}
where $u: \mathbb{R}^3\to\mathbb{R}^3$ is the velocity vector, $p:\mathbb{R}^3\to\mathbb{R}$ is the pressure and $\nu>0$ is the viscosity constant. These equations are invariant under the scaling $u(x)\to \lambda u(\lambda x)$ and $p(x)\to \lambda^2 p(\lambda x)$ for any $\lambda>0$. It is natural to study solutions which are invariant under this scaling. 
 For such solutions, $u$ is $(-1)$-homogeneous and $p$ is $(-2)$-homogeneous, and we call them $(-1)$-homogeneous solutions according to the homogeneity of $u$. In general, a function $f$ is said to be $(-\alpha)$-homogeneous if $f(x)= \lambda^\alpha f(\lambda x)$ for any $\lambda>0$. 

Let $x=(x_1,x_2,x_3)$ be the standard Euclidean coordinates, and $(r,\theta, \phi)$ be the spherical coordinates, where $r$ is the radial distance from the origin, $\theta\in [0, \pi]$ is the polar angle measured from the positive $x_3$-axis to the vector, and $\phi\in [0, 2\pi)$ is the meridian angle about the $x_3$-axis. A vector field $u$ can be written as
$
	u = u_r e_r + u_\theta e_{\theta} + u_\phi e_{\phi},
$
where
\[
	e_r = \left(
	\begin{matrix}
		\sin\theta\cos\phi \\
		\sin\theta\sin\phi \\
		\cos\theta
	\end{matrix} \right), \hspace{0.5cm}
	e_{\theta} = \left(
	\begin{matrix}
		\cos\theta\cos\phi \\
		\cos\theta\sin\phi \\
		-\sin\theta	
	\end{matrix} \right), \hspace{0.5cm}
	e_{\phi} = \left(
	\begin{matrix}
		-\sin\phi \\ \cos\phi \\ 0
	\end{matrix} \right).
\]
A vector field $u$ is said to be \emph{axisymmetric} if the components $u_r$, $u_{\theta}$ and $u_{\phi}$ are independent of $\phi$, and it is called \emph{no-swirl} if $u_{\phi}\equiv 0$. For $(-1)$-homogeneous solutions $(u, p)$ in $\mathbb{R}^3\setminus\{0\}$, (\ref{NS}) can be reduced to a system of partial differential equations for $(u, p)$ on $\mathbb{S}^2$. For any set $\Omega\subset\bS^2$, a $(-1)$-homogeneous solution $(u, p)$ on $\Omega$ is understood to have been extended to the set $\{x\in \mathbb{R}^3\mid x/|x|\in \Omega\}$ so that $u$ is $(-1)$-homogeneous and $p$ is $(-2)$-homogeneous. We use this convention throughout the paper. 

In 1944, Landau \cite{Landau1944} discovered a 3-parameter family of explicit $(-1)$-homogeneous solutions to the stationary Navier-Stokes equations in $C^\infty(\mathbb{R}^3\setminus\{0\})$, see also \cite{Slezkin1934} and \cite{Squire1951}. These solutions, now known as \emph{Landau solutions}, are axisymmetric with no swirl and have exactly one point singularity at the origin. The expressions of Landau solutions $(u^b, p^b)$ are given by 
\be
  u^b=\frac{2}{r}(\frac{A^2-1}{(A-\cos\theta)^2}-1)e_r-\frac{2\sin\theta}{r(A-\cos\theta)}e_{\theta}, \quad p^b=-\frac{4(A\cos\theta-1)}{r^2(A-\cos\theta)^2}, 
\ee
for some constant $A>1$. The net force associated with Landau solutions is a Dirac mass 
$f=b\delta_0$, where $b=(0, 0, \beta)$ and $\beta=16\pi (A+\frac{1}{2}A^2\ln \frac{A-1}{A+1}+\frac{4A}{3(A^2-1)})$. 
Landau interpreted his solutions as a model of some jet discharged from the origin.

It is natural to ask whether there exist other $(-1)$-homogeneous solutions of (\ref{NS}) except Landau solutions. Tian and Xin proved in \cite{TianXin1998} that all nonzero $(-1)$-homogeneous, axisymmetric solutions of (\ref{NS}) in $C^\infty(\mathbb{R}^3\setminus\{0\})$ must be Landau solutions. \v{S}ver\'{a}k established the following result in 2006:
\begin{theorem}[\cite{Sverak2011}]\label{thm:Sverak}
	All $(-1)$-homogeneous nonzero solutions of (\ref{NS}) in $C^2(\mathbb{R}^3\setminus\{0\})$ are Landau solutions.
\end{theorem}
In the same paper, \v{S}ver\'{a}k also proved that there is no nonzero $(-1)$-homogeneous solution of the incompressible stationary Navier-Stokes equations in $C^{2}(\mathbb{R}^n\setminus\{0\})$ for $n\ge 4$. In dimension two, he characterized all such solutions satisfying a zero flux condition. Homogeneous solutions of (\ref{NS}) have also been studied in various other works, see, e.g. \cite{CKPW2022, Goldshtik1960, Gusarov2020, JLL2021, KT2021, LZZ2023, PP1985a, PP1985b, PP1985c, Serrin1972, Slezkin1934, Squire1951, Wang1991, Yatseyev1950, ZZ2021}. There have also been works on homogeneous solutions of Euler's equations, see \cite{Abe2024, LS2015, Shvydkoy2018} and the references therein. 

A natural follow-up problem is to study $(-1)$-homogeneous solutions of (\ref{NS}) with possible finitely many singular rays. To this end, we consider solutions of (\ref{NS}) in $C^{2}(\mathbb{S}^2\setminus\{P_1,..., P_m\})$ where $P_1,..., P_m$ are finitely many points on $\mathbb{S}^2$.
Such solutions are subject to less restrictive symmetry and regularity conditions, making related questions—such as their classification and stability—more challenging. Below are several central problems in the study of such solutions:

\medskip

\noindent\textbf{Questions:}
\begin{enumerate}
  \item \emph{Existence and classification}: Let $m\ge 1$ and $P_1,..., P_m$ be arbitrary points on $\mathbb{S}^2$. Does there exist any $(-1)$-homogeneous solution of (\ref{NS}) in $C^{2}(\mathbb{S}^2\setminus\{P_1,..., P_m\})$? If yes, how can we classify such solutions?
  \item \emph{Singularity behavior}: What is the asymptotic behavior of $(-1)$-homogeneous solutions of (\ref{NS}) near a singular point $P$ on $\bS^2$? 
  \item \emph{Vanishing viscosity limit}: What is the limiting behavior of $(-1)$-homogeneous solutions of (\ref{NS}) as the viscosity $\nu\to 0$? How do they converge or not converge to solutions of Euler equations? 
  \item \emph{Asymptotic stability}: Are $(-1)$-homogeneous solutions stable under perturbations? In particular, under what conditions such solutions are asymptotically stable, and in what norms or perturbation settings? 
\end{enumerate}

The classification of $(-1)$-homogeneous solutions of (\ref{NS}) in $C^{\infty}(\bS^2)$ (Theorem \ref{thm:Sverak}) in general does not extend to solutions with singularities on $\bS^2$. To study the above problems, a natural first step is to investigate solutions that are axisymmetric. 

The existence and classification of $(-1)$-homogeneous axisymmetric solutions in $C^2(\mathbb{S}^2\setminus\{S, N\})$, where $S$ is the south pole and $N$ is the north pole, were studied in \cite{LLY2018a}-\cite{LLY2019a}. 
In \cite{LLY2018a}, all such no-swirl solutions in $C^2(\mathbb{S}^2\setminus\{S\})$ were classified, and 
 existence and nonexistence results of such solutions with nonzero swirl in $C^2(\mathbb{S}^2\setminus\{S\})$ were established. In \cite{LLY2018b},  a complete classification of all $(-1)$-homogeneous axisymmetric no-swirl solutions in $C^2(\mathbb{S}^2\setminus\{S, N\})$ was provided. In \cite{LLY2019a}, existence and non-existence results for $(-1)$-homogeneous axisymmetric solutions in $C^2(\mathbb{S}^2\setminus\{S, N\})$ with nonzero swirl were established.
%
Additionally, \cite{LLY2018a} derived the asymptotic expansions of all local $(-1)$-homogeneous axisymmetric solutions of (\ref{NS}) near a singular ray. In \cite{LLY2024}, an optimal removable singularity result for $(-1)$-homogeneous solutions near a singular ray was established without assuming axisymmetry. 
The vanishing viscosity limit of $(-1)$-homogeneous axisymmetric no-swirl solutions of (\ref{NS}) in $C^2(\mathbb{S}^2\setminus\{S, N\})$ was analyzed in \cite{LLY2019b}. 
Furthermore, the asymptotic stability of the least singular axisymmetric no-swirl solutions under $L^2$-perturbations was proved in \cite{LY2021}, while the asymptotic stability of Landau solutions under $L^2$-perturbations was established in \cite{Karch2011}. 

In this paper, we provide an exposition to recent developments in the study of $(-1)$-homogeneous solutions of (\ref{NS}) with singular rays, along with make some remarks and discussions. 
In Section \ref{sec_2}, we  discuss the above mentioned results in \cite{LLY2018a}-\cite{LY2021}. 
In Section \ref{sec_3}, we further investigate the behavior of $(-1)$-homogeneous axisymmetric no-swirl solutions of (\ref{NS}), and provide graphical illustrations to offer some intuitive description of them. 

\section{Recent study on homogeneous solutions of (\ref{NS})}\label{sec_2}

\subsection{Existence and classification}\label{sec2_1}

The existence and classification  of $(-1)$-homogeneous axisymmetric solutions of (\ref{NS}) in $C^2(\mathbb{S}^2\setminus\{S, N\})$ were studied in  \cite{LLY2018a, LLY2018b, LLY2019a}. 
For these solutions, the system (\ref{NS}) on $\mathbb{S}^2$ takes the form
\begin{equation}\label{eq_homo}
	\left\{
	\begin{aligned}
		& \nu \frac{d^2 \ur}{d\theta^2} + (\nu \ctthe - \uthe) \frac{d \ur}{d\theta} + \ur^2 +\uthe^2 +\uphi^2 +2p = 0; \\
		& \frac{d}{d\theta} (\frac{1}{2}\uthe^2 - \nu \ur + p) = \cot\theta \uphi^2; \\
		& \nu \frac{d}{d \theta}(\frac{d \uphi}{d \theta} + \cot\theta \uphi) - \uthe (\frac{d \uphi}{d \theta} + \cot\theta \uphi) =0; \\
		& \ur + \frac{d \uthe}{d \theta} + \ctthe \uthe = 0. \quad \quad \mbox{(divergence free condition)}\\   
	\end{aligned}
	\right.
\end{equation}
If we further assume that $u$ has no swirl, i.e. $u_{\phi}\equiv 0$, then (\ref{NS}) is reduced to 
\begin{equation}\label{eq_homo_1}
  \nu (\frac{du_{\theta}}{d\theta}-\cot\theta u_{\theta})-\frac{1}{2}u_{\theta}^2=\frac{a\cos^2\theta+b\cos\theta+c}{\sin^2\theta}
\end{equation}
for some constants $a, b, c\in\R$, 
and $u_{r}$ and $p$ can be determined by 
\begin{equation}\label{eq_up}
  u_r = - \frac{d u_{\theta}}{d \theta} - u_{\theta} \cot \theta, \quad p=\nu u_r-\frac{1}{2}u_{\theta}^2 + \textrm{const}, 
  \quad \textrm{ on }\mathbb{S}^2. 
\end{equation}
By rescaling $\tilde{u}=u/\nu$, (\ref{NS}) is equivalent to the same equation of $\tilde{u}$ with $\nu=1$. For convenience,  \cite{LLY2018a, LLY2018b, LLY2019a} focused on the case $\nu=1$ when studying the existence and classification of solutions. In this paper, we present all the results for general $\nu$ to more clearly illustrate the behavior of $u$ in terms of $\nu$. 
In our study, the following change of variables and functions played an important role.
\begin{equation}\label{Va}
  y:=\cthe, \quad U_r := u_r \sthe, \quad U_\theta := u_\theta \sthe, \quad U_\phi:= u_\phi \sthe. 
\end{equation}
Then (\ref{NS}) for $(-1)$-homogeneous axisymmetric solutions is converted to 
\begin{equation}\label{eq_NSx2}
  \left\{
  \begin{aligned}
	  & \nu (1-y^2) U_\theta' + 2\nu y U_\theta +\frac{1}{2} U_\theta^2 + \int \int \int \frac{2 U_\phi(s) U_\phi'(s)}{1-s^2} ds dt dl = b_1 y^2 + b_2 y + b_3, \\
	  & \nu (1-y^2) U_\phi'' + U_\theta U_\phi' = 0,
  \end{aligned}
  \right. 
\end{equation}
where $b_1,b_2, b_3$ are some constants, $y\in (-1,1)$, 
and the derivatives are in $y$. The divergence free condition is equivalent to 
$
  u_r=U'_{\theta}(y). 
$
This change of variables had been used by Serrin \cite{Serrin1972}. 

\subsubsection{Classification of axisymmetric no-swirl solutions}\label{sec2_1_1}

\noindent \textbf{a. Classification of axisymmetric no-swirl solutions in $\bS^2\setminus\{S\}$}
\medskip

In \cite{LLY2018a}, we classified all $(-1)$-homogeneous axisymmetric no-swirl solutions of (\ref{NS}) in $C^{\infty}(\bS^2\setminus\{S\})$ as a two-parameter family with explicit expressions. 
Let $\tau, \sigma\in \R$, and
\be
  \mathcal{I}_{\nu} := \{(\tau,\sigma)\in\mathbb{R}^2\mid \tau\le 2\nu, \sigma < \nu-\tau/4\}\cup \{(\tau,\sigma)\mid \tau\ge 2\nu, \sigma =\tau/4\},
\ee
\begin{thm}[\cite{LLY2018a}]\label{thmLLY1_1}
  For every $ (\tau, \sigma)\in \mathcal{I}_{\nu}$, there exists a unique $(-1)$-homogeneous axisymmetric no-swirl solution $(u^{\tau, \sigma}, p^{\tau, \sigma})\in C^{\infty}(\mathbb{S}^2\setminus\{S\})$ to (\ref{NS}) on $\bS^2\setminus\{S\}$, up to a constant translation of $p$, satisfying $\lim_{\theta\to \pi^-} u^{\tau, \sigma}_{\theta}\sin\theta=\tau$ and $\lim_{\theta\to 0^+}u^{\tau, \sigma}_{\theta}/\sin\theta=\sigma$.
  Moreover, these are all the axisymmetric no-swirl solutions in $C^2(\mathbb{S}^2\setminus\{S\})$.	The explicit expressions of these solutions are as follows. 
	\begin{equation}\label{eq_temp11}
		u^{\tau, \sigma}_\theta = 
		\left\{
		\begin{array}{ll}
			\frac{\displaystyle \nu (1-\cos\theta)}{\displaystyle \sin\theta}\left(1- b - \frac{\displaystyle 2 b(1-2\sigma/\nu-b)}{\displaystyle(1-2\sigma/\nu+b) (\frac{1+\cos\theta}{2})^{-b} + 2\sigma/\nu-1 + b} \right), & \tau<2\nu, \\
			\frac{\displaystyle \nu(1-\cos\theta)}{\displaystyle \sin\theta} \left( 1+ \frac{\displaystyle 2 (1-2\sigma/\nu) }{ \displaystyle (1-2\sigma/\nu) \ln \frac{1+\cos\theta}{2} - 2 } \right), & \tau=2\nu, \\
			\frac{\displaystyle \nu(1+b)(1-\cos\theta)}{\displaystyle \sin\theta}, & \tau>2\nu, 
		\end{array}
		\right. 
	\end{equation}
	where $(\tau, \sigma)\in \mathcal{I}_{\nu}$ and $b := |1-\tau/(2\nu)|$. 
\end{thm}
For each $(\tau, \sigma)\in \mathcal{I}_{\nu}$, $u^{\tau, \sigma}_{\theta}$ is given by (\ref{eq_temp11}), $u^{\tau, \sigma}_{\phi}=0$, $u^{\tau, \sigma}_{r}$ and $p^{\tau, \sigma}$ can be determined by (\ref{eq_up}). 
In particular, Landau solutions are given by $\{(u^{\tau, \sigma}, p^{\tau, \sigma})| \tau=0, \sigma\in (-\infty, 0)\cup (0,\nu)\}$, where $u^{0, \sigma}_{\theta}=\frac{2\nu \sin\theta}{\cos\theta+2\nu/\sigma-1}$. See Figure \ref{fig: tausigma} for the parameter set $\mathcal{I}_{\nu}\subset \R^2$. 
\begin{figure}[!htb]
	\centering
	\includegraphics[height=5.5cm]{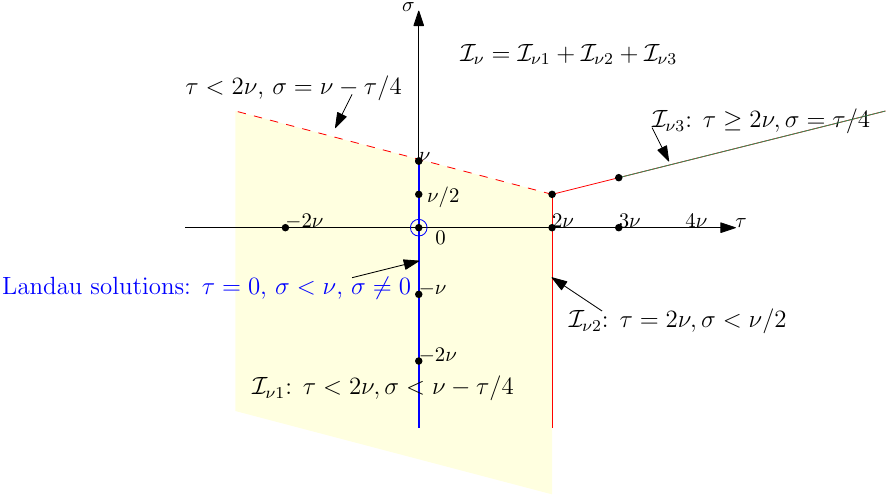}
	\vspace{-0.5cm}
	\caption{\small{Parameter set $\mathcal{I}_{\nu}\subset \R^2$ of $(\tau, \sigma)$.}}
	\label{fig: tausigma}
\end{figure}

\noindent \textbf{b. Classification of axisymmetric no-swirl solutions in $\bS^2\setminus\{S, N\}$}
 
\medskip
 
The classification of $(-1)$-homogeneous axisymmetric no-swirl solutions of (\ref{NS}) in $C^2(\bS^2\setminus\{S, N\})$ was established in \cite{LLY2018b}. 
Using the change of variables (\ref{Va}), for $\nu>0$, equation (\ref{eq_homo_1}) for no-swirl solutions is converted to 
\be \label{eq:NSE}
	\nu (1-y^2)U_\theta' + 2\nu y U_\theta + \frac{1}{2}U_\theta^2 = P_c(y):=c_1(1-y)+c_2(1+y)+c_3(1-y^2), 
\ee
for $-1<y<1$. 
For solutions with two singularities $S$ and $N$ on $\bS^2$, (\ref{eq:NSE}) cannot be solved explicitly as in Theorem \ref{thmLLY1_1}. Equation (\ref{eq:NSE}) is a Ricatti-type ODE, which can be converted to a second order linear ODE,  whose solutions can be expressed in terms of hypergeometric functions. 
We will discuss such representations of solutions of (\ref{eq:NSE}) in Section \ref{sec_3}. 
In \cite{LLY2018b} with Y.Y. Li, the authors analyzed the nonlinear equation (\ref{eq:NSE}) and classified all $(-1)$-homogeneous axisymmetric no-swirl solutions of (\ref{NS}) on $\bS^2\setminus\{S, N\}$ as a four-parameter family. 
For $\nu\ge 0$, define
$
  \bar{c}_3(c_1,c_2; \nu)=-\frac{1}{2}(\sqrt{\nu^2+c_1}+\sqrt{\nu^2+c_2})(\sqrt{\nu^2+c_1}+\sqrt{\nu^2+c_2}+2\nu),
$
and let
\begin{equation}\label{eq_J}
  J_\nu:=\{c\in \mathbb{R}^3\mid c=(c_1,c_2,c_3), c_1\ge -\nu^2, c_2\ge -\nu^2, c_3\ge \bar{c}_3(c_1,c_2;\nu)\}.
\end{equation}
In \cite{LLY2018b}, it was proved that there exist solutions of (\ref{eq:NSE}) in $C^{1}(-1, 1)$ if and only if $c\in J_{\nu}$. 
For each $c$ in $J_{\nu}$, there exist an upper solution $U^+_{\theta}(c)$ and a lower solution $U^-_{\theta}(c)$ in $C^1(-1, 1)$ of (\ref{eq:NSE}), and all other solutions  in $C^1(-1, 1)$  lie between them, foliating the region bounded by $U^{\pm}_{\theta}(c)$. 
Base on Theorem 1.1 and Theorem 1.3 in \cite{LLY2018b}, we summarize the following results. 

\begin{thm}[\cite{LLY2018b}]\label{thmLLY2_0}
	For each $\nu>0$, there exist $\gamma^-(c),\gamma^+(c)\in C^0(J_{\nu}, \mathbb{R})$, satisfying $\gamma^-(c)<\gamma^+(c)$ if $c_3>\bar{c}_3(c_1,c_2)$, and $\gamma^-(c)=\gamma^+(c)$ if $c_3=\bar{c}_3(c_1,c_2)$, such that for every $c$ in $J_{\nu}$ and $\gamma^-(c)\le \gamma \le \gamma^+(c)$, there exists a unique $(-1)$-homogeneous axisymmetric no-swirl solution $(u^{c, \gamma}, p^{c, \gamma})\in C^{\infty}(\mathbb{S}^2\setminus\{S,N\})$ of (\ref{NS}), 
	up to a constant translation of $p$, such that $U_{\theta}^{c, \gamma}=u_{\theta}^{c,\gamma}\sin\theta$ satisfies (\ref{eq:NSE}) on $(-1, 1)$, and $u_{\theta}^{c,\gamma}(\frac{\pi}{2})=\gamma$. 
  Moreover, these are all $(-1)$-homogeneous axisymmetric no-swirl solutions of (\ref{NS}) on $\mathbb{S}^2\setminus\{S,N\}$. 
  In particular, if $c_3=\bar{c}_3(c_1,c_2; \nu)$, then there exists a unique solution of (\ref{eq:NSE})  in $C^1(-1, 1)$, where the corresponding 
  \begin{equation}\label{eqLLY2_0_1}
    u_{\theta}^*=\frac{(\nu+\sqrt{\nu^2+c_1})(1-\cos\theta)+(-\nu-\sqrt{\nu^2+c_2})(1+\cos\theta)}{\sin\theta}. 
  \end{equation}
  If $c\notin J_{\nu}$, then there  exists no solution of (\ref{eq:NSE}) in $C^{1}(-1, 1)$. 
\end{thm}
	
Denote $u^{\pm}(c)=u^{c, \gamma^{\pm}(c)}$, and the corresponding $U^{\pm}_{\theta}=U_{\theta}^{c, \gamma^{\pm}}$. We refer to $u^{+}(c)$ (or $U^+_{\theta}(c)$) as the upper solution and $u^{-}(c)$ (or $U^-_{\theta}(c)$) as the lower solution. 
Any $(-1)$-homogeneous axisymmetric no-swirl solution $u^{c, \gamma}\in C^2(\bS^2\setminus\{S, N\})$ of (\ref{NS}) satisfies $u^-_{\theta}(c)\le u^{c, \gamma}_{\theta}\le u^+_{\theta}(c)$. 
Set
\begin{equation}\label{eq_I}
  I_{\nu}:= \{(c, \gamma)\in \mathbb{R}^4 \mid c\in J_{\nu}, \gamma^-(c)\leq \gamma \leq \gamma^+(c) \}.
\end{equation}
Then all $(-1)$-homogeneous axisymmetric no-swirl solutions of (\ref{NS}) in $C^2(\mathbb{S}^2\setminus \{S, N\})$ are classified as the four-parameter family $\{(u^{c, \gamma}, p^{c, \gamma})\mid (c, \gamma)\in I_{\nu}\}$.
There is a 1-1 correspondence between $u^{c,\gamma}$ and points in the four dimensional surface $I_{\nu}$.
The Landau solutions correspond to $u^{c,\gamma}$ with $c=0$ and $\gamma\in(-2\nu, 2\nu)\setminus\{0\}$, where $u^{0, \gamma}_{\theta}=\frac{2\nu \sin\theta}{\cos\theta+2\nu/\gamma}$. 
The solutions $u^{\tau, \sigma}$ in $C^{\infty}(\mathbb{S}^2\setminus\{S\})$ correspond to those $u^{c,\gamma}$ with $c_2=0$, $c_1=-2c_3$ and $\gamma^-(c)<\gamma\le \gamma^+(c)$. Likewise, the solutions in $C^{\infty}(\mathbb{S}^2\setminus\{N\})$ correspond to $u^{c,\gamma}$ with $c_1=0$, $c_2=-2c_3$ and $\gamma^-(c)\le \gamma< \gamma^+(c)$.
	
The detailed behavior of $\{(u^{c, \gamma}, p^{c, \gamma})\mid (c, \gamma)\in I_{\nu}\}$ was studied in \cite{LLY2018b}. Some of the results are summarized below in terms of the corresponding $U^{c, \gamma}_{\theta}(y)$.  
For $\nu>0$ and $c\in J_{\nu}$, define
\begin{equation}\label{eq_2}
\begin{array}{lll}
	&\tau_{1}:=2\nu-2\sqrt{\nu^2+c_1}, \quad &\tau_{2}:=2\nu+2\sqrt{\nu^2+c_1}, \\
	& \tau_{1}':=-2\nu-2\sqrt{\nu^2+c_2}, \quad &\tau_{2}':=-2\nu+2\sqrt{\nu^2+c_2}.
\end{array}
\end{equation}
\begin{prop}[\cite{LLY2018b}]\label{prop2_1}
  Let $\nu>0$ and $(c,\gamma)\in I_{\nu}$. \\
  (i) If $c_3>\bar{c}_3 (c_1,c_2; \nu)$, then $\gamma^-(c)<\gamma^+(c)$, and $U_{\theta}^{c,\gamma} < U_{\theta}^{c,\gamma'}$ in $(-1,1)$ for any $\gamma^-(c) \leq \gamma < \gamma' \leq \gamma^+(c)$. Moreover,	$\{(y,z)\mid -1<y<1, U_{\theta}^{-} (c)\leq z \leq U_{\theta}^{+} (c) \} = \bigcup_{\gamma\in [\gamma^-(c),\gamma^+(c)]} \{(y,U_{\theta}^{c,\gamma}) \mid -1<y<1\}$. 

  \medskip

	\noindent (ii) For any $(c,\gamma)\in I_{\nu}$, $U_{\theta}^{c, \gamma}\in C^{\infty}(-1, 1)\cap C[-1, 1]$. Moreover, 
	\begin{equation}\label{eqLLY2_1}
		U_{\theta}^{c,\gamma}(-1) = \left\{
		\begin{array}{ll}
			\tau_2 & \mbox{if } \gamma=\gamma^+(c), \\
			\tau_1, & \mbox{otherwise},
		\end{array}
		\right. \quad
		U_{\theta}^{c,\gamma}(1) = \left\{
		\begin{array}{ll}
			\tau_1', & \mbox{if } \gamma=\gamma^-(c), \\
			\tau_2', & \mbox{otherwise}.
		\end{array}
		\right.
	\end{equation}
  In particular, $U_{\theta}^+(c)$ is the unique solution of (\ref{eq:NSE}) on $(-1, -1+\delta)$ for any $\delta>0$ satisfying $U_{\theta}(-1)=\tau_2$, and $U_{\theta}^-(c)$ is the unique solution of (\ref{eq:NSE}) on $(1-\delta, 1)$ for any $\delta>0$ satisfying $U_{\theta}(1)=\tau_1'$.
	
	\medskip
	
	\noindent (iii) There exist some $\delta>0$, and sequences of constants $\{a_n\}_{n=1}^{\infty}$ and $\{b_n\}_{n=1}^{\infty}$, depending only on $c$, such that $U_{\theta}^+(c)=\tau_2+\sum_{n=1}^{\infty}a_n(1+y)^n$ is real analytic in $(-1, -1+\delta)$, and $U_{\theta}^-(c)=\tau_1'+\sum_{n=1}^{\infty}b_n(1-y)^n$ is real analytic in $(1-\delta, 1)$. 
\end{prop}

In \cite{LLY2018b}, the authors also analyzed  the smoothness properties of $\gamma^{\pm}(c)$ with respect to $c$ for $c$ in different regions of $J_{\nu}$, as well as  the smoothness of $U^{c,\gamma}_{\theta}$ with respect to $(c,\gamma)$ for $(c,\gamma)$ in different regions of $I_{\nu}$. 
The details and proofs of these results can be found in \cite{LLY2018b}.

\subsubsection{Existence and nonexistence of axisymmetric solutions with nonzero swirl}

In \cite{LLY2018a} and \cite{LLY2019a},  the existence and nonexistence of $(-1)$-homogeneous axisymmetric solutions of (\ref{NS}) with nonzero swirl were studied, considering $u\in C^{2}(\mathbb{S}^2\setminus\{S\})$ in \cite{LLY2018a} and $u\in C^{2}(\mathbb{S}^2\setminus\{S, N\})$ in \cite{LLY2019a}. 
By Theorem \ref{thmLLY1_1}, all $(-1)$-homogeneous axisymmetric no-swirl solutions of (\ref{NS}) on $\bS^2\setminus\{S\}$ are classified as the two parameter family $\{(u^{\tau, \sigma}, p^{\tau, \sigma})\mid (\tau, \sigma)\in \mathcal{I}_{\nu}\}$. 
In \cite{LLY2018a}, it was proved that for every $(\tau, \sigma)\in \mathcal{I}_{\nu}\cap \{\tau<3\nu\}$, there exists a one-parameter family of axisymmetric solutions with nonzero swirl in $C^{2}(\mathbb{S}^2\setminus\{S\})$, emanating from $u^{\tau,\sigma}$ in the $e_{\phi}$ direction. 
In contrast, for $(\tau, \sigma)\in \mathcal{I}_{\nu} \cap\{\tau>3\nu\}$, such solutions do not exist.  
Nearby $\{(u^{\tau, \sigma}, p^{\tau, \sigma})\mid (\tau, \sigma)\in \mathcal{I}_{\nu}, \tau>3\nu\}$, there is no $(-1)$-homogeneous axisymmetric solution with nonzero swirl. 

The existence result was established using perturbation methods. Define $b_{\tau,\sigma}(\theta)=-\frac{1}{\nu} \int_{\frac{\pi}{2}}^{\theta}(u_{\theta})_{\tau,\sigma}dt$, and set 
\[
   v_{\tau, \sigma}=\frac{\nu}{ \sin\theta}\left(\int_{0}^{\theta}e^{-b_{\tau,\sigma}(t)}\sin t dt\right) e_{\phi}. 
\]
Then $v_{\tau, \sigma}$ is a solution of the linearized equation of (\ref{NS}) at $(u^{\tau, \sigma}, p^{\tau, \sigma})$. 
We decompose $\mathcal{I}_{\nu}=\mathcal{I}_{\nu 1}\cup \mathcal{I}_{\nu 2}\cup \mathcal{I}_{\nu 3}$, where 
\[
  \mathcal{I}_{\nu 1}=\mathcal{I}_{\nu}\cap \{\tau<2\nu, \sigma <\nu-\frac{\tau}{4}\}, \mbox{ }
  \mathcal{I}_{\nu 2}=\mathcal{I}_{\nu}\cap \{\tau=2\nu, \sigma <\frac{\nu}{2}\}, \mbox{ }
  \mathcal{I}_{\nu 3}= \mathcal{I}_{\nu}\cap \{\tau\ge 2\nu, \sigma =\frac{\tau}{4}\}. 
\]

\begin{thm}[\cite{LLY2018a}]\label{thmLLY1_2}
  Let $K$ be a compact subset of one of the four sets $\mathcal{I}_{\nu 1}$, $\mathcal{I}_{\nu 2}$, $\mathcal{I}_{\nu 3}\cap\{2\nu< \tau<3\nu \}$ and $\mathcal{I}_{\nu 3}\cap \{\tau=2\nu \}$, then there exist $\delta=\delta(K)>0$, and $(u,p)\in C^{\infty}(K\times (-\delta,\delta)\times (\mathbb{S}^2\setminus\{S\}))$ such that for every $(\tau,\sigma, \beta)\in K\times(-\delta,\delta)$, $(u,p)(\tau,\sigma,\beta; \cdot)\in C^{\infty}(\mathbb{S}^2\setminus\{S\})$ satisfies (\ref{NS}) in $\mathbb{R}^3\setminus\{(0,0,x_3)|x_3\le 0\}$, with nonzero swirl if $\beta\ne 0$, and $||\left(\sin\frac{\theta+\pi}{2}\right)(u(\tau, \sigma,\beta)-u_{\tau, \sigma})||_{L^{\infty}(\mathbb{S}^2\setminus\{S\})}\to 0$ as $\beta\to 0$. Moreover, $ \frac{\partial }{\partial \beta}u(\tau, \sigma,\beta)|_{\beta=0}=v_{\tau,\sigma}$.
    
  On the other hand, for $(\tau,\sigma)\in \mathcal{I}_3\cap\{\tau>3\nu \}$, there does not exist any sequence of solutions $\{u^i\}$ of (\ref{NS}) in $C^{\infty}(\mathbb{S}^2\setminus\{S\})$, with nonzero swirl, such that $||\left(\sin\frac{\theta+\pi}{2}\right)(u^i-u_{\tau, \sigma})||_{L^{\infty}(\mathbb{S}^2\setminus\{S\})}\to 0$ as $i\to \infty$. 
\end{thm}

This result in particular provided the first examples of $(-1)$-homogeneous solutions with nonzero swirl of (\ref{NS}) and a singular ray exhibiting finite ordered singularity behavior, i.e. $\limsup_{|x|=1, x\to S} |u(x)| |\textrm{dist} (x,S)|^k<\infty$ for some $k>0$. 
A more detailed and stronger version of Theorem \ref{thmLLY1_2}, including a uniqueness result, can be found 
 in \cite{LLY2018a} (see Theorem 4.1, Theorem 4.2 and  Theorem 4.3 there).

The existence part of Theorem \ref{thmLLY1_2} was established using perturbation methods via implicit function theorems. The key step is to construct appropriate Banach spaces and operators, where the change of variables (\ref{Va}) and (\ref{eq_NSx2}) was used. Three different sets of spaces are needed for $(\tau, \sigma)$ in $\mathcal{I}_{\nu 1}$, $\mathcal{I}_{\nu 2}$ and $\mathcal{I}_{\nu 3}$, 
where the singular behaviors of $u_{\tau, \sigma}$ are different. 
In particular, for some $(\tau, \sigma)$ on the boundary of $\mathcal{I}_{\nu}$, the spaces involve logarithmic functions. The nonexistence part follows from analyzing the asymptotic behavior of the solutions near the singularity $S$ on $\bS^2$,  see Theorem \ref{thmLLY1_3} for details. 

Using a similar approach,  \cite{LLY2019a} established the existence and nonexistence results for $(-1)$-homogeneous axisymmetric solutions of \eqref{NS} with nonzero swirl  in $C^{2}(\mathbb{S}^2\setminus\{N,S\})$ nearby the no-swirl solutions $(u^{c,\gamma}, p^{c,\gamma})$. Detailed statements can be found in \cite{LLY2019a}. 

\subsubsection{Remarks on the existence of non-axisymmetric solutions}
\label{sec2_1_3}

In this subsection, we make a brief discussion about the existence and nonexistence of non-axisymmetric $(-1)$-homogeneous solutions of (\ref{NS}) with singular rays. 
Theorem \ref{thm:Sverak} has ruled out the 
existence of non-axisymmetric $(-1)$-homogeneous solutions of (\ref{NS}) in $C^{\infty}(\bS^2)$. 
But what about solutions with singularities on $\bS^2$? 

Some special examples of non-axisymmetric $(-1)$-homogeneous solutions can be inferred from \cite{Sverak2011}, which 
connects $(-1)$-homogeneous solutions of (\ref{NS}) to the conformal geometry on $\bS^2$. 
In \cite{Sverak2011}, it was proved that if $u\in C^{\infty}(\bS^2)$ is a $(-1)$-homogeneous solution of (\ref{NS}), then there exists some function $\varphi\in C^{\infty}(\bS^2)$ satisfying 
\begin{equation}\label{eq:liouville:S2}
	-\triangle_{\mathbb{S}^2} \varphi + 2 = 2 e^{\varphi} \quad \mbox{on } \mathbb{S}^2, 
\end{equation}
such that $u=\nabla_{\mathbb{S}^2} \varphi-\Delta_{\mathbb{S}^2} \varphi e_r$ on $\mathbb{S}^2$. 
Conversely, if $\varphi$ is a solution to (\ref{eq:liouville:S2}), then $u=\nabla_{\mathbb{S}^2} \varphi-\Delta_{\mathbb{S}^2} \varphi e_r$ is a solution to (\ref{NS}) after being extended to a $(-1)$-homogeneous vector field in $\mathbb{R}^3$. 
By the classical Liouville formula (see \cite{Liouville1853, CW1994}), the solutions of (\ref{eq:liouville:S2}) takes the form 
\begin{equation}\label{eq:liouville:xi}
  \displaystyle \varphi= \ln \frac{|f'(z)|^2(1+|z|^2)^2}{(1+|f(z)|^2)^2},
\end{equation}
where $f$ is a locally univalent meromorphic function and $z\in \bC$ is the stereographic projection of $x\in\bS^2$. 
This enables us to construct $(-1)$-homogeneous solutions of (\ref{NS}), including non-axisymmetric ones, by selecting appropriate functions $f$. 

Let $u$ be a $(-1)$-homogeneous solution of (\ref{NS}) on $\mathbb{S}^2\setminus\{S, N\}$ constructed via the Liouville formulas above. If $u\in C^{\infty}(\bS^2)$, then the corresponding $f$ must be of the form $f=az$ for some constant $a$, and $u$ is a Landau solution (see \cite{Sverak2011}). If $u$ is axisymmetric but may have singularities on $\bS^2$, it was proved in \cite{LLY2024} that $f=a z^{\alpha}$ for some $a\in \mathbb{C}\setminus\{0\}$ and $\alpha\in\mathbb{R}\setminus\{0\}$. 
Non-axisymmetric solutions can be constructed by choosing suitable forms of $f$, such as $f=ae^{bz}$, or $f=\int_0^z(t-z_1)^{l_1}\cdots (t-z_m)^{l_m-1}dt$, where the latter one generates solutions of (\ref{NS}) with multiple singularities on $\bS^2$. 

Further investigation into the existence and nonexistence of other non-axisymmetric solutions is of interest. For instance, a natural question is whether nonaxisymmetric solutions exist nearby known axisymmetric solutions. 
  
\subsection{Asymptotic behavior of solutions near singular rays}

To study $(-1)$-homogeneous solutions of (\ref{NS}) with finitely many singularities on $\bS^2$, it is helpful to first analyze the asymptotic behavior of solutions near an isolated singularity $P$ on $\bS^2$. In particular, we focus on the asymptotic behavior and the classifications of solutions satisfying $u=O(1/ {\rm dist}(x, P)^k)$ for some positive integer $k$.  The asymptotic behavior of local $(-1)$-homogeneous axisymmetric solutions of (\ref{NS}) in $B_{\delta}(S)\cap \mathbb{S}^2\setminus\{S\}$, for any $\delta>0$, was analyzed in \cite{LLY2018a} and \cite{LLY2018b}. In \cite{LLY2024},  an optimal removable singularity result was established for $(-1)$-homogeneous solution $u\in C^2(B_{\delta}(S)\cap \mathbb{S}^2\setminus\{S\})$. Below we give an exposition of  these results and make some further discussion. 

\subsubsection{Asymptotic behavior of homogeneous axisymmetric solutions} 

In \cite{LLY2018a} and \cite{LLY2018b},  the asymptotic expansions of all local $(-1)$-homogeneous axisymmetric solutions of (\ref{NS}) in $B_{\delta}(S)\cap \mathbb{S}^2\setminus\{S\}$, for any $\delta>0$, were established. A summary of these results was provided in \cite{LLY2024}, which is recalled below. 
\begin{thm}[\cite{LLY2024}] \label{thmLLY1_3}
	Let $\delta>0$, and $u\in C^2(\mathbb{S}^2 \cap B_\delta(S)\setminus\{S\})$ be a $(-1)$-homogeneous axisymmetric solution of (1). Denote $x'=(x_1,x_2)$. Then $\tau:= \lim_{x\in\mathbb{S}^2,x\to S} |x'| u_\theta$ exists and is finite, and $u = O(1/\big| |x'| \ln |x'| \big|^2)$. Moreover, 
		
	{\rm (i)} If $\tau \ge 3\nu$, then $|x'| u_\phi$ must be a constant, and $|x'| u_\theta$ and $u_r$ must be real analytic functions in $1+\cos\theta$ near $S$ on $\mathbb{S}^2$. 
	
	{\rm (ii)} If $2\nu < \tau < 3\nu$, then either $|x'|u_\phi \equiv $constant, or $\lim_{x\in\mathbb{S}^2,x\to S} |x'|^{\tau/\nu-1} u_\phi$ exists and is finite and not zero. Moreover, $\lim_{x\in\mathbb{S}^2,x\to S} |x'|^{2\tau/\nu-4} u_r$ exists and is finite. 
	
	{\rm (iii)} If $\tau=2\nu$, then $\eta:= \lim_{x\in \mathbb{S}^2, x\to S} (|x'|u_\theta-2\nu)\ln |x'|$ exists and is $0$ or $2\nu$. 
	
	- When $\eta=0$, then $\lim_{x\in\mathbb{S}^2,x\to S} |x'|^{\epsilon} u_r = 0$ for any $\epsilon>0$. Either $|x'| u_\phi$ is a constant, or $\lim_{x\in\mathbb{S}^2,x\to S} |x'| u_\phi$ exists and is finite and not zero. 
	
	- When $\eta=2\nu$, then $\lim_{x\in\mathbb{S}^2,x\to S} |x'|^2 \big| \ln |x'| \big|^2 u_r = -2\nu$, and $\lim_{x\in\mathbb{S}^2,x\to S}|x'|u_\phi$ exists and is finite. 
	
	{\rm (iv)} If $\tau < 2\nu$ and $\tau \not= 0$, then $\lim_{x\in\mathbb{S}^2,x\to S} |x'|u_\phi$ and $\lim_{x\in\mathbb{S}^2,x\to S} |x'|^{\max\{\tau/\nu,0\}} u_r$ both exist and are finite. 
	
	{\rm (v)} If $\tau=0$, then $\sigma:=\lim_{x\in\mathbb{S}^2,x\to S} |x'|u_\phi$ and 
	$\lim_{x\in\mathbb{S}^2,x\to S}u_r/\ln |x'|$ both exist and are finite. 
	 Moreover, $\tilde{u}:=u-\sigma/|x'|e_{\phi}$ is also a solution of (\ref{NS}), and 
	 $\lim_{x\in\mathbb{S}^2,x\to S} |\tilde{u}|/\big| \ln |x'| \big|$ exists and is finite. 
\end{thm}

\medskip
Note (i) implies that if $\tau\ge 3$ and $u$ is a $(-1)$-homogeneous axisymmetric solution in $C^2(\bS^2\setminus\{S\})$, then $u_{\phi}\equiv 0$. This in particular provided the proof of the nonexistence part of Theorem \ref{thmLLY1_2}. 

Similar results to those in Theorem \ref{thmLLY1_3} hold for solutions $u\in C^2(\mathbb{S}^2 \cap B_\delta(N)\setminus\{N\})$. In an ongoing project \cite{LLY2024prep} with Y.Y. Li, we have established the full asymptotic expansion for all $(-1)$-homogeneous axisymmetric no-swirl solutions of (\ref{NS}) in $\mathbb{S}^2\cap B_{\delta}(S)\setminus\{S\}$. 
These expansions take the form of three types of infinite series for $\tau$ in three different sets. 
We have also proved the convergence of these series solutions and shown that all local $(-1)$-homogeneous axisymmetric no-swirl solutions of (\ref{NS}) in $\mathbb{S}^2\cap B_{\delta}(S)\setminus\{S\}$ are completely classified by them. 
We expect that similar results hold for axisymmetric solutions with nonzero swirl, while the expansions are more intricate. We plan to study this problem in the near future.

\subsubsection{Removable singularity for $(-1)$-homogeneous solutions on $\bS^2$}\label{sec2_2_2}

The next step is to study the asymptotic behavior of general $(-1)$-homogeneous solutions of (\ref{NS}) near possible singular rays, without assuming axisymmetry. 
In \cite{LLY2024},  the following removable singularity problem was considered: \emph{For local $(-1)$-homogeneous solutions (not necessarily axisymmetric) of (\ref{NS}) near a potential singularity on $\bS^2$, under what condition is the singularity removable?} 

There has been much study on the behavior of solutions of (\ref{NS}) near isolated singularities in $\R^3$, see e.g. \cite{DE1970, Shapiro1974, Shapiro1976a, Shapiro1976b, CK2000, KK2006, KS2011, MT2012, NP2000, Sverak2011}. 

The asymptotic behavior of axisymmetric solutions, as described in Theorem \ref{thmLLY1_3}, suggests that the least singular behavior of a solution near a singularity $S$ on $\bS^2$ is of the order of $\ln {\rm dist}(x,S)$. 
Therefore, a natural removable singularity condition is $u=o(\ln {\rm dist}(x,S))$. Denote $B_{\delta}(S) := \{ x \in \mathbb{R}^3 \mid {\rm dist}(x,S)<\delta\}$ for $\delta>0$. 
The following result has been established in \cite{LLY2024}. 
\begin{thm}[\cite{LLY2024}]\label{thmLLY5_1}
	Let $\delta>0$, 
	$(u, p)\in C^{2}(\mathbb{S}^2\cap B_{\delta}(S)\setminus\{S\})$ be a $(-1)$-homogeneous solution of (\ref{NS}). If 
  $\lim_{x\in \mathbb{S}^2, x\to S}\frac{|u(x)|}{\ln {\rm dist}(x,S)}=0$, then $(u, p)$ can be extended as a $C^2$ function in $\mathbb{S}^2\cap B_{\delta}(S)$. 
\end{thm}

  The above removable singularity result is optimal in the following sense: For any $\alpha>0$, there exists a $(-1)$-homogeneous axisymmetric no-swirl solution $(u, p)\in C^{\infty}(\mathbb{S}^2\setminus\{S, N\})$ of (\ref{NS}), such that $\lim_{x\in \mathbb{S}^2, x\to S}|u(x)|/\ln {\rm dist}(x, S)=$\\$\lim_{x\in \mathbb{S}^2, x\to N}|u(x)|/\ln {\rm dist}(x, N) = - \alpha$. Examples of such solutions can be found in \cite{LLY2018b} and \cite{LLY2024}, see example \ref{ex1} for a special case. 

\medskip

In view of Theorem \ref{thmLLY1_3}, 
all $(-1)$-homogeneous axisymmetric solutions of (\ref{NS}) in $C^{\infty}(\mathbb{S}^2\setminus\{S, N\})$ are of the following three mutually exclusive types:
\begin{enumerate}
	\item[]Type 1. Landau solutions, satisfying $\sup_{|x|=1}|u(x)|<\infty$; 
	\item[]Type 2. Solutions satisfying $0<\limsup_{|x|=1,x'\to 0}|u(x)|/| \ln |x'||<\infty$; 
	\item[]Type 3. Solutions satisfying $\limsup_{|x|=1,x'\to 0}|x'||u(x)|>0$. 
\end{enumerate}
Here we denote $x'=(x_1, x_2)$. 
Several examples of solutions with isolated singularities on $\bS^2$, along with discussions of their asymptotic behaviors, were provided in \cite{LLY2024} and summarized below. 

\medskip

\noindent\textbf{a. $(-1)$-homogeneous axisymmetric solutions of (\ref{NS}) on $\bS^2\setminus\{S, N\}$} The existence and classification of these solutions are discussed in Section \ref{sec2_1}. 
Among them, the only Type 2 solutions are a specific family of axisymmetric no-swirl solutions in $C^{\infty}(\bS^2\setminus\{S, N\})$, given by $\{(u^{c, \gamma}, p^{c, \gamma})\mid (c, \gamma)\in I_{\nu}, c_1=c_2=0, c_3>-4\nu^2, c_3\ne 0, \gamma^-<\gamma<\gamma^+\}$ (see Section \ref{sec2_1_1} for the definition of $u^{c, \gamma}$, and Example \ref{ex1} for some special solutions). 
These solutions satisfy $\lim_{x\in \mathbb{S}^2, x\to P}|u^{c,\gamma}|/\ln \textrm{dist}(x, P)=-2|c_3|$, where $P=S$ or $N$. Landau solutions correspond to $\{(u^{c, \gamma}, p^{c, \gamma})\mid c=0, \gamma^-<\gamma<\gamma^+\}$. 
In \cite{LLY2024}, it was shown that all other $(-1)$-homogeneous axisymmetric solutions $u\in C^2(\mathbb{S}^2\setminus\{S, N\})$ are of Type 3. 

\medskip

\noindent\textbf{b. Serrin's solutions} In a pioneering work \cite{Serrin1972} concerning a representative model for tornado-like flows, Serrin studied $(-1)$-homogeneous axisymmetric solutions of (\ref{NS}) in the upper half space $\mathbb{R}^3_+:=\mathbb{R}^3\cap\{x_3>0\}$, with a singular ray along the positive $x_3$-axis and some boundary conditions on $\partial \mathbb{R}^3_+$. Notably, the solutions formulated in \cite{Serrin1972} exhibit behavior distinct from Landau solutions. Near the north pole $N$ on $\bS^2$, these solutions satisfy $u_{\theta}=O(|x'| \ln |x'|)$, $u_r=O(\ln |x'|)$
and $0<\lim_{x\in \mathbb{S}^2, x\to N}|x'| |u_{\phi}|<\infty$. In particular, Serrin's solutions are of Type 3. 

\medskip \noindent\textbf{c. Solutions given by Liouville formulas}
These solutions are discussed in Section \ref{sec2_1_3}. Aside from Landau solutions, all examples obtained via the Liouville formulas are of Type 3. 
There are no Type 2 solutions given by Liouville formulas. 

\medskip

\noindent\textbf{d. Solutions that are also solutions of Euler's equation} There is a family of $(-1)$-homogeneous axisymmetric no-swirl solutions of (\ref{NS}) that also satisfy the Euler's equations. These solutions take the form $u_{\theta}=\frac{a\cos\theta+b}{\sin\theta}$, $u_r=a$, $u_{\phi}=0$, $p=-\frac{a^2+b^2+2ab\cos\theta}{2\sin^2\theta}$ for some constants $a, b\in \R$. 
These solutions are also of Type 3. 

\subsubsection{Further discussion on singular behavior}
Consider general $(-1)$-homogeneous solutions of (\ref{NS}), we raise the following questions. 
\begin{que}
  Let $\delta>0$, $P\in \bS^2$ and $u\in C^{\infty}(B_{\delta}(P)\cap \bS^2\setminus\{P\})$ be a $(-1)$-homogeneous solution of (\ref{NS}). What are the possible asymptotic behaviors of $u$ near $P$?
\end{que}
\begin{que}
  How can we classify $(-1)$-homogeneous solutions of (\ref{NS}) on $\bS^2$ that satisfy $u=O\left(|\textrm{dist}(x,S)|^{-k}\right)$ for some integer $k\ge 1$?
\end{que}
By Theorem \ref{thm:Sverak} (\cite{Sverak2011}),
all solutions in $C^{\infty}(\bS^2)$ are Landau solutions, which we refer to as Type 1 solutions. By the removable singularity result (Theorem \ref{thmLLY5_1}), for any $P\in \bS^2$, a $(-1)$-homogeneous solution in $C^{\infty}(B_{\delta}(P)\cap \bS^2\setminus\{P\})$ satisfying $u=o(\ln \text{dist} (x, P))$ can be smoothly extended across $P$. The next natural class of interest is the family of $(-1)$-homogeneous solutions exhibiting Type 2 behavior, i.e. solutions satisfying $0<\limsup_{|x|=1,x'\to 0}|u(x)|/| \ln |x'||<\infty$. 
A natural question is: 
\emph{Does there exist a nonaxisymmetric $(-1)$-homogeneous solution exhibiting Type 2 behavior? Under what conditions must a Type 2 solution  in $C^{\infty}(\bS^2\setminus\{S, N\})$ be axisymmetric? }

As shown in \cite{LLY2024},  any Type 2 axisymmetric solution on $\bS^2\setminus\{S, N\}$ must be no-swirl. The only Type 2 axisymmetric solutions in $C^{\infty}(\bS^2\setminus\{S, N\})$ are the family $\{u^{c, \gamma}\mid (c, \gamma)\in I_{\nu}, c_1=c_2=0, c_3> -4\nu^2, c_3\ne 0, \gamma^-<\gamma<\gamma^+\}$.  In particular, there is no $(-1)$-homogeneous axisymmetric solution
$u\in C^2(\mathbb{S}^2\setminus\{P\})$ of (\ref{NS}) satisfying $0<\limsup_{|x|=1,x\to P}|u(x)|/|\ln {\rm dist}(x, P)|<\infty$, where $P=S$ or $N$. 

If nonaxisymmetric solutions are considered, then the classification becomes much more involved. So far, the only nonaxisymmetric solutions we know are those constructed via Liouville formulas (see Section \ref{sec2_1_3}). 
These solutions are of the form $u=\nabla_{\mathbb{S}^2} \varphi-\Delta_{\mathbb{S}^2} \varphi e_r$ on $\mathbb{S}^2\setminus\{S, N\}$ where $\varphi$ is given by (\ref{eq:liouville:xi}), generated from some meromorphic function $f$. 
Using this expression, it can be shown that such solutions do not exhibit Type 2 behavior. 

\subsection{Vanishing viscosity limit of homogeneous solutions}\label{sec2_3}%

Let $\nu=0$ in (\ref{NS}), then (\ref{NS}) reduces to the $3$D incompressible stationary Euler equations 
\begin{equation}\label{Euler}
  v\cdot \nabla v +\nabla q = 0, \quad \dive\textrm{ } v=0.
\end{equation}
A fundamental question in fluid dynamics is whether solution of Navier-Stokes equations converge to the solutions of Euler equations as $\nu\to 0$. 
Building on the classification of all $(-1)$-homogeneous axisymmetric no-swirl solutions of (\ref{NS}) on $\bS^2\setminus\{S, N\}$, \cite{LLY2019b} studied the vanishing viscosity limit of sequences of these solutions. 

The classification of all $(-1)$-homogeneous axisymmetric solutions of the Euler equations (\ref{Euler}) were obtained in \cite{TianXin1998}. 
Let $J_0$ be defined by (\ref{eq_J}) with $\nu=0$, and $\partial'J_0=(\partial J_0\cap \{c_1=0\textrm{ or }c_2=0\})\setminus\{0\}$, where $\partial J_0$ denotes the boundary of $J_0$. Let $P_c(x)$ be the second order polynomial given in (\ref{eq:NSE}).
Let $v^{\pm}_c=v^{\pm}_{c, r}\vec{e}_r + v^{\pm}_{c, \theta} \vec{e}_\theta$, where
$$
  v^{\pm}_{c, \theta} (r, \theta, \varphi)=\pm \frac { \sqrt{ 2P_c(\cos\theta) } }{r \sin\theta}, \quad 
  v^{\pm}_{c, r}(r, \theta, \varphi)=\pm \frac{P_c'(\cos\theta)}{r\sqrt{2P_c(\cos\theta)}}, 
$$
and
$
  q_c(r, \theta, \varphi)=-\frac{1}{2r^2}(P_c''(\cos\theta)+\frac{2P_c(\cos\theta)}{\sin^2\theta}).
$
Then $\{(v^{\pm}_c, q_c)\ |\ c\in \mathring J_0\cup \partial' J_0\}$ is the set of all $(-1)$-homogeneous axisymmetric no-swirl solutions of (\ref{Euler}) in $C^{\infty}(\mathbb S^2 \setminus\{S, N\})$. 

Let $c\in J_0$, $\nu_k\to 0$, and $(c_k, \gamma_k)\in I_{\nu_k}$ be a sequence such that $c_k\to c$. The question is whether the corresponding $(-1)$-homogeneous axisymmetric no-swirl solutions $(u_{\nu_k}^{c_k, \gamma_k}, p_{\nu_k}^{c_k, \gamma_k})$ (as identified in Theorem \ref{thmLLY2_0}) of (\ref{NS}) converge to one of the solutions $(v^{\pm}_c, q^{\pm}_c)$ of (\ref{Euler}) as $\nu_k\to 0$ and $c_k\to c$. 
For $c$ and $c_k$ in different parts of $J_0$ and $J_{\nu_k}$, the limiting behaviors of $(u_{\nu_k}^{c_k, \gamma_k}, p_{\nu_k}^{c_k, \gamma_k})$ 
 are different. 
In particular, the behaviors of the upper and lower solutions $\{(u_{\nu_k}^{\pm}, p_{\nu_k}^{\pm})\}$ are different from other solutions. 
In most cases, $(u_{\nu_k}^{\pm}, p_{\nu_k}^{\pm})$ will $C^m_{loc}$ converge to $(v_c^{\pm}, q_c^{\pm})$ on the entire sphere, 
while for other solutions, boundary layer or interior later behaviors occur. Specifically, or every latitude circle, there exist sequences $(u_{\nu_k}, p_{\nu_k})$ that $C^m_{loc}$ converge to different solutions of (\ref{Euler}) 
 on the spherical caps above and below the latitude circle respectively, separated by a transition layer. 
In \cite{LLY2019b}, the thickness of the transition layers and the convergence rates of the solutions as $\nu_k\to 0$ were analyzed in detail.  
Below we present some partial results for the case when $c_k=c$ in $\mathring{J}_{0}$ (the interior of $J_0$). For more details of the vanishing viscosity behavior of all $(-1)$-homogeneous axisymmetric no-swirl solutions of (\ref{NS}) on $\bS^2\setminus\{S, N\}$, we refer the readers to \cite{LLY2019b}. 
\begin{thm}\label{thm1_0}
  (i) There exist $(-1)$-homogeneous axisymmetric no-swirl solutions
  \newline $\{(u^{\pm}_{\nu}(c), p^{\pm}_{\nu}(c))\}_{0< \nu\le 1}$ of (\ref{NS}), belonging to $C^{0}(\mathring{J}_{\nu}\times (0,1], C^m(\mathbb{S}^2\setminus(B_{\epsilon}(S)\cup B_{\epsilon}(N))))$ for every integer $m\ge 0$, such that for every compact subset $K\subset \mathring{J}_{0}$, and every $\epsilon>0$, there exists some constant $C$ depending only on $\epsilon, K$ and $m$, such that
  \[
    ||(u^{\pm}_{\nu}(c), p^{\pm}_{\nu}(c))-(v^{\pm}_c, q_c)||_{C^m(\mathbb{S}^2\setminus\{B_{\epsilon}(S)\cup B_{\epsilon}(N)\})}\le C\nu, \quad c\in K.
  \]
  (ii) For every $0<\theta_0<\pi$, there exist $(-1)$-homogeneous axisymmetric no-swirl solutions $\{(u_{\nu}(c, \theta_0), p_{\nu}(c, \theta_0))\}_{0<\nu\le 1}$ of (\ref{NS}), belonging to $C^{0}(\mathring{J}_{\nu}\times(0,1]\times(0,\pi), C^m(\mathbb{S}^2\setminus(B_{\epsilon}(S)\cup B_{\epsilon}(N))))$ for every integer $m\ge 0$, such that for every compact subset $K\subset \mathring{J}_0$, and every $\epsilon>0$, there exists some constant $C$ depending on $\epsilon, K$ and $m$, such that
  \[
   \begin{split}
    & ||(u_{\nu}(c, \theta_0), p_{\nu}(c, \theta_0))-(v^{-}_c, q_c)||_{C^m(\mathbb{S}^2\cap \{\theta_0+\epsilon<\theta<\pi-\epsilon\})}\\
    &+||(u_{\nu}(c, \theta_0), p_{\nu}(c, \theta_0))-(v^{+}_c, q_c)||_{C^m(\mathbb{S}^2\cap\{\epsilon<\theta<\theta_0-\epsilon\})}\le C\nu, \quad c\in K.
  \end{split}
  \]
\end{thm}
For every $c$ in $\mathring{J}_0$, $P_c>0$ on $[-1,1]$, and $v^{+}_c \ne v^{-}_c$ on $\mathbb{S}^2\cap\{\theta=\theta_0\}$. The limit functions in Theorem \ref{thm1_0} (ii) have jump discontinuities across the circle $\{\theta=\theta_0\}$. In Section \ref{sec_3}, we will present graphs and streamlines of selected examples to illustrate the behavior of solutions as the viscosity tends to zero. 
We also plan to extend the study about vanishing viscosity limit to axisymmetric solutions with nonzero swirl obtained in \cite{LLY2018a} and \cite{LLY2019a}. 

\subsection{Asymptotic stability of homogeneous solutions}\label{sec2_4}

Consider the $3$D incompressible evolutionary 
Navier-Stokes equations 
\begin{equation}\label{NSE}
  \left\{
    \begin{array}{ll}
      u_t-\nu\Delta u+(u\cdot \nabla) u+\nabla p=f,& \textrm{ in } \mathbb{R}^3\times (0,\infty), \\
      \dive u=0, & \textrm{ in } \mathbb{R}^3\times (0,\infty), \\
      u(x, 0)=u_0(x), & 
    \end{array}
  \right.
\end{equation}
with initial velocity $u_0(x)$ 
 and external force $f=f(x, t)$. 
When $f=0$, it is a classical result that there exists a weak solution to (\ref{NSE}) satisfying $\lim_{t\to \infty}\|u(\cdot, t)\|_{L^2(\R^3)}=0$, demonstrating the asymptotic stability of the trivial solution 
(see, e.g., \cite{Cannone2004, Kajikiya1986, Kato1984, Lemarie2002, Leray1934, Masuda1984, Schonbek1985, Teman1977}). 
The asymptotic stability has also been studied for other steady states \cite{CK2004, CKPW2022, Karch2011, KPS2017, BW2025}. For Landau solutions, the external force takes the form $f=b\delta_0$ for some constant $b$ as mentioned in Section \ref{sec:intro}. 
In \cite{Karch2011}, it was proved that small Landau solutions are asymptotically stable under $L^2$-perturbations. The $L^p$-stability of Landau solutions for $p\ge 3$ has been studied by \cite{LZZ2023, ZZ2021}. 

Since the 
 solutions   obtained in \cite{LLY2018a, LLY2018b, LLY2019a} exhibit  different behavior from Landau solutions, it is natural to explore their asymptotic stability or instability. 
This question was studied in  \cite{LY2021} for the class of $(-1)$-homogeneous axisymmetric solutions of (\ref{NS}) exhibiting Type 2 behavior. These solutions are the family $\{(u^{c, \gamma}, p^{c, \gamma})\mid (c, \gamma)\in M_{\nu}\}$ among the solutions classified in Theorem \ref{thmLLY2_0}, where $M_{\nu}:=\{(c,\gamma)\mid c_1=c_2=0, c_3>-4\nu^2, \gamma^-<\gamma<\gamma^+\}$. 
Each of these solutions satisfies $u^{c,\gamma}=O(\ln |x'|)$ and $p^{c,\gamma}=O(\ln |x'|)$ as $x'\to 0$ on $\bS^2$, and they exhibit the least singular behavior among all $(-1)$-homogeneous solutions except Landau solutions. 
The corresponding external force of such a solution is given by 
$f^{c, \gamma}=(4\pi c_3 \ln |x_3| \partial_{x_3}\delta_{(0,0,x_3)}-b^{c,\gamma}\delta_{0})e_3$, where $b^{c,\gamma}=2\pi\int_{-1}^{1}\left(y|U_{\theta}'|^2-\frac{2+y^2}{1-y^2}\nu U_{\theta}-\frac{y}{1-y^2}U_{\theta}^2 \right)dy$. 

For any fixed $(c, \gamma)\in M_{\nu}$, 
let $u$ be a solution of (\ref{NSE}) with $f=f^{c, \gamma}$. Let  
$w=u-u^{c,\gamma}$, $\pi=p-p^{c, \gamma}$ and $w_0=u_0-u^{c,\gamma}$. 
Then $w$ satisfies the perturbed system 
\begin{equation}\label{eqS}
  \left\{
  \begin{array}{ll}
    w_t-\nu\Delta w+(w\cdot \nabla) w+(w\cdot \nabla) u^{c,\gamma}+(u^{c,\gamma}\cdot \nabla) w+\nabla \pi=0, & 
    \textrm{ in } \mathbb{R}^3\times (0,\infty),\\ 
    \dive w=0, & \textrm{ in } \mathbb{R}^3\times (0,\infty),\\
    w(x,0)=w_0(x). &
  \end{array}
  \right.
\end{equation}
\begin{thm}[\cite{LY2021}]\label{thmS}
  There exists some $\mu_0>0$, such that for any $c=(0,0,c_3)$, $|(c,\gamma)|<\mu_0$ and $w_0\in L^2_{\sigma}(\mathbb{R}^3)$, there exists a weak solution $w$ of (\ref{eqS}) in $\mathbf{X}:=L^{\infty}([0, \infty), L^2_{\sigma}(\mathbb{R}^3))\cap L^2([0,\infty), \dot{H}^1_{\sigma})$. Moreover, $w$ is weakly continuous from $[0,\infty)$ to $L^2_{\sigma}(\mathbb{R}^3)$, and satisfies 
  $\|w(t)\|_{L^2(\mathbb{R}^3)}^2+\int_{s}^{t}\|\nabla \otimes w(\tau)\|_{L^2(\mathbb{R}^3)}^2 d\tau\le \|w(s)\|_{L^2(\mathbb{R}^3)}^2$  
  for almost all $s\ge 0$, including $s=0$ and all $t\ge s$. 
  Furthermore, any such weak solution satisfies $\displaystyle \lim_{t\to \infty}\|w(t)\|_{L^2(\mathbb{R}^3)}=0$. If in addition $w_0\in L^p(\mathbb{R}^3)\cap L^2_{\sigma}(\mathbb{R}^3)$ for some $\frac{6}{5}<p<2$, then there exists some constant $C>0$, depending only on $(c,\gamma)$, $n, p$ and $\|w_0\|_{p}$, such that $\|w(t)\|_2\le Ct^{-\frac{3}{2}(\frac{1}{p}-\frac{1}{2})}$, for all $t>0$.
\end{thm}
Here $L^2_{\sigma}(\mathbb{R}^3))$ denote the space of all divergence free functions in $L^2(\R^3)$. The proof of Theorem \ref{thmS} relies on semigroup theory and follows the framework in \cite{Karch2011}. The key step is to establish the energy estimates 
$|\int_{\mathbb{R}^3}(v\cdot\nabla u^{c,\gamma})\cdot w dx| \le K\|\nabla w\|_{L^2}\|\nabla v\|_{L^2}$,
for some sufficiently small constant $K$, and any divergence free $v, w\in C_c^{\infty}(\mathbb{R}^3)$. 
When $u^{c,\gamma}$ is replaced by a small Landau solution $u^b$, for which $|\nabla u^b|=O(1/|x|^2)$,  the energy estimate follows from the classical Hardy's inequality (see \cite{Karch2011}). 
However, the solutions $u^{c,\gamma}$ are more singular, 
which satisfy $|\nabla u^{c, \gamma}|=O(\frac{1}{|x||x'|})$ and Hardy's inequality does not work in this context. 
In \cite{LY2021}, the energy estimate was proved 
by establishing a {\it new} weighted Hardy type inequality:
\begin{equation}\label{eqH_3}
  \int_{\mathbb{R}^n}\frac{|w(x)|^2}{|x\|x'|}dx\le C \int_{\mathbb{R}^n}|\nabla w(x)|^2\frac{|x'|}{|x|}dx, \textrm{ for all }w\in \dot{H}^1(\mathbb{R}^n), 
\end{equation}
where $x'=(x_1,...,x_{n-1})$ for dimensions $n\ge 2$, and $C>0$ is a constant depending only on $n$. 
Note that (\ref{eqH_3}) is stronger than the classical Hardy's inequality (which is (\ref{eqH_3}) with $|x'|$ replaced by $|x|$). This type of inequalities are of independent research interest. 
In \cite{LY2023}, inequality (\ref{eqH_3}) was further extended into a family of rather general anisotropic Caffarelli-Kohn-Nirenberg inequalities. 
The optimal constant and extreme functions have been studied recently by \cite{BC2025, HY2025}. 
The $L^p$-stability of Type 2 solutions $u^{c,\gamma}$ 
has been studied in \cite{ZZ2023} for $p\ge 3$. 

A natural next objective is to study the asymptotic stability of $(-1)$-homogeneous solutions exhibiting Type 3 behavior. We begin with the least singular case, where $|\bar{u}|=O(1/|x'|)$ on $\mathbb{S}^2$. These solutions are more singular than Type 2 solutions, and the stability proof in \cite{LY2021} 
 does not apply. New approaches or analytical tools will be needed. 
As a first step, we have identified the force of all the $(-1)$-homogeneous axisymmetric no-swirl solutions in $C^2(\bS^2\setminus\{S, N\})$ in an ongoing project \cite{LLY2024prep}.

\section{Some remarks and graphs of homogeneous axisymmetric no-swirl solutions}\label{sec_3}
In this section, we make some remarks and discussions on $(-1)$-homogeneous axisymmetric no-swirl solutions $(u, p)$ of (\ref{NS}) on $\bS^2\cap\{\theta_0<\theta<\theta_1\}$, for some $0\le \theta_0<\theta_1\le \pi$. Recall for such solutions, (\ref{NS}) is reduced to (\ref{eq:NSE}) of $U_{\theta}$, with $c=(c_1, c_2, c_3)\in \R^3$. We refer to $(u, p)$ as a \emph{global solution} if $\theta_0=0$ and $\theta_1=\pi$, in which case the corresponding $U_{\theta}$ satisfies (\ref{eq:NSE}) on $(-1, 1)$. Otherwise, when $\theta_0>0$ or $\theta_1<\pi$, we say $(u, p)$ is a \emph{local solution}. 
By Theorem \ref{thmLLY2_0}, all global solutions 
are classified as a four-parameter family $\{u^{c, \gamma}\mid (c,\gamma)\in I_{\nu}\}$, where $I_{\nu}$ is defined by (\ref{eq_I}). In this section, we investigate the behavior of local solutions and demonstrate graphs of some examples of both global and local solutions. To facilitate this analysis, we transform 
(\ref{eq:NSE}) into a second-order linear differential equation. In particular, if $c_1, c_2\ge -\nu^2$, 
(\ref{eq:NSE}) can be converted to a hypergeometric differential equation, enabling more explicit analysis and graphical illustration.

\subsection{Representation of 
solutions} 

As mentioned in Section \ref{sec2_1_1}, equation (\ref{eq:NSE}) is a Ricatti-type ODE and can be converted to a second-order linear ODE. Let $U_{\theta}$ be a solution of (\ref{eq:NSE}) on some interval $(y_0, y_1)\subset (-1, 1)$. Set 
$w(y):=e^{\int_{y_0}^y\frac{U_{\theta}(s)}{2\nu (1-s^2)}ds}$, then $w\ne 0$ on $(y_0, y_1)$ and $U_{\theta}$ can be expressed by 
\be\label{eqH_0_1}
  U_{\theta}=U_{\theta}[w]:=2\nu (1-y^2)\frac{w'(y)}{w(y)}.
\ee
By computation, (\ref{eq:NSE}) is equivalent to 
\be\label{eqH_0_2}
  2\nu^2(1-y^2)^2w''(y)-(c_1(1-y)+c_2(1+y)+c_3(1-y^2))w(y)=0, 
\ee
for $y_0<y<y_1$. 
By standard ODE theory, (\ref{eqH_0_2}) always have two linearly independent solution $w_1, w_2\in C^2(-1, 1)$, and its general solution is given by $w=C_1w_1+C_2w_2$ for constants $C_1, C_2$. 

Let $\bar{y}\in (-1, 1)$ and $\gamma\in \R$, consider the initial value problem 
\be\label{eqN_1}
  \nu (1-y^2)U_\theta' + 2\nu y U_\theta + \frac{1}{2}U_\theta^2 =c_1(1-y)+c_2(1+y)+c_3(1-y^2), \quad 
  U_{\theta}(\bar{y})=\gamma.
\ee 
Let $w_1, w_2$ be a pair of fundamental solutions of (\ref{eqH_0_2}). Denote 
\be\label{eqN_1_2}
  \mu:=\frac{\gamma}{2\nu (1-\bar{y}^2)},\quad \textrm{ and } D:= -\frac{w_2'(\bar{y})-\mu w_2(\bar{y})}{w_1'(\bar{y})-\mu w_1(\bar{y})} \textrm{ if }w_1'(\bar{y})\ne \mu w_1(\bar{y}).
\ee
Recall we have defined $\tau_1, \tau_2, \tau_1', \tau_2'$ in (\ref{eq_2}). 
\begin{lem}\label{lem3_1}
  Let $\nu>0$, $c=(c_1, c_2, c_3)\in \R^3$, $\bar{y}\in (-1, 1)$ and $\gamma\in \R$. Then there exist some $y_0, y_1$ satisfying $-1\le y_0<\bar{y}<y_1\le 1$, 
  and a unique solution $U_{\theta}\in C^{1}(y_0, y_1)$ of (\ref{eqN_1}), given by (\ref{eqH_0_1}), with 
  \be\label{eq3_1_1}
    w(y)=\left\{
    \begin{array}{ll}
      Dw_1(y)+w_2(y), & \textrm{ if }w_1'(\bar{y})\ne \mu w_1(\bar{y}),\\
      w_1(y),  &\textrm{ if } w_1'(\bar{y})= \mu w_1(\bar{y}), 
    \end{array}
    \right.
  \ee
  where $w\ne 0$ on $(y_0, y_1)$, $w_1, w_2$ are two linearly independent fundamental
  solutions of (\ref{eqH_0_2}), and $\mu $ and $D$ are defined by (\ref{eqN_1_2}). Moreover, these are all the solutions of (\ref{eq:NSE}) on some interval $(y_0, y_1)\subset (-1, 1)$. 
\end{lem}
\begin{proof}
  Let $w$ be defined by (\ref{eq3_1_1}) and $U_{\theta}$ be defined by (\ref{eqH_0_1}). Note $w\in C^{\infty}(-1, 1)$. We show that there exist some $\epsilon>0$, such that $U_{\theta}$ is the unique solution of (\ref{eqN_1}) in $(\bar{y}-\epsilon, \bar{y}+\epsilon)$. 

  We first show that there is some $\epsilon>0$ such that $w\ne 0$ in $(\bar{y}-\epsilon, \bar{y}+\epsilon)$, for which we only need to show $w(\bar{y})\ne 0$. Indeed, if $w_1'(\bar{y})\ne \mu w_1(\bar{y})$, then by (\ref{eqN_1_2}) and (\ref{eq3_1_1}), 
  \[
    w(\bar{y})=Dw_1(\bar{y})+w_2(\bar{y})=\frac{w_1'(\bar{y})w_2(\bar{y})-w_2'(\bar{y})w_1(\bar{y})}{w_1'(\bar{y})-\mu w_1(\bar{y})}.
  \]
  As $w_1, w_2$ are linearly independent solutions of (\ref{eqH_0_2}), we have $w_1'(\bar{y})w_2(\bar{y})-w_2'(\bar{y})w_1(\bar{y})\ne 0$, 
  and thus $w(\bar{y})\ne 0$. 
  When $w_1'(\bar{y})= \mu w_1(\bar{y})$, we have $w(y)=w_1(y)$. 
  If $w_1(\bar{y})=0$, then $w_1'(\bar{y})=0$ as well. By standard ODE theory, we have $w_1\equiv 0$, which is a contradiction since $w_1$ is a fundamental solution of (\ref{eqH_0_2}). 
  So $w(\bar{y})\ne 0$ in both cases, 
  and there exists some $\epsilon>0$, such that $w\ne 0$ and $U_{\theta}$ given by (\ref{eqH_0_1}) is well-defined in $(\bar{y}-\epsilon, \bar{y}+\epsilon)$. Since $w$ satisfy (\ref{eqH_0_2}), a direct computation shows $U_{\theta}$ satisfies the ODE in (\ref{eqN_1}) in $(\bar{y}-\epsilon, \bar{y}+\epsilon)$. By the definition of $U_{\theta}$, $w$ and $D$, we also have $U_{\theta}(\bar{y})=\gamma$ by direct computation. 
  The uniqueness of the solution of (\ref{eqN_1}) also follows from standard ODE theories. 
\end{proof}
\begin{lem}\label{lem3_3}
 Let $\nu>0$, $c=(c_1, c_2, c_3)\in \R^3$, $-1\le y_0<y_1\le 1$. Assume $U_{\theta}$ is a solution of (\ref{eq:NSE}) whose domain in $(-1, 1)$ is $(y_0, y_1)$. Then\\
  (a) If $y_0=-1$, then  $U_{\theta}(-1):=\lim_{y\to -1^+} U_{\theta}$ exists and is finite, and $U_{\theta}(-1)=\tau_1$ or $\tau_2$. 
  If $y_1=1$, then $U_{\theta}(1):=\lim_{y\to 1^-} U_{\theta}$ exists and is finite, and $U_{\theta}(1)=\tau_1'$ or $\tau_2'$. \\
  (b) If $y_0>-1$, then $\lim_{y\to y_0^+} U_{\theta}=+\infty$. If $y_1<1$, then $\lim_{y\to y_1^-} U_{\theta}=-\infty$. 
\end{lem}
\begin{proof}
Let $U_{\theta}\in C^1(y_0, y_1)$ be a solution of (\ref{eq:NSE}) and its domain 
on $(-1, 1)$ is $(y_0, y_1)$. Let $\bar{y}\in (y_0, y_1)$ and $\gamma=U_{\theta}(\bar{y})$. 
By Lemma \ref{lem3_1}, there exists some $w\in C^{\infty}(-1, 1)$ satisfying (\ref{eqH_0_2}) on $(-1, 1)$, such that $w\ne 0$ on $(y_0, y_1)$, and $U_{\theta}$ is given by (\ref{eqH_0_1}).   
By Theorem 1.3 in \cite{LLY2018a}, 
  if $y_0=-1$, then $U_{\theta}(-1)=\lim_{y\to - 1^+} U_{\theta}$ exists and is finite, and $\lim_{y\to -1^+}(1-y^2)U'_{\theta}=0$. By this and (\ref{eq:NSE}), we have 
  $U_{\theta}(-1)=\tau_1$ or $\tau_2$, where $\tau_1, \tau_2$ are defined by (\ref{eq_2}). Similarly, if $y_1=1$, then $U_{\theta}(1)=\lim_{y\to 1^-} U_{\theta}$ exists and is finite, 
  and $U_{\theta}(1)=\tau_1'$ or $\tau_2'$, where $\tau_1', \tau_2'$ are defined by (\ref{eq_2}). So (a) holds.
   
  Now we prove (b). We only show $\lim_{y\to y_0^+} U_{\theta}=+\infty$ if $y_0>-1$. The other part can be proved similarly. 
  Note $U_{\theta}$ given by (\ref{eqH_0_1}) is well-defined and is a solution of (\ref{eq:NSE}) at $y\in (-1, 1)$ if and only if $w(y)\ne 0$. 
  If $y_0>-1$, then $U_{\theta}$ is not well-defined at $y_0$ and 
    $w(y_0)=0$. If $w'(y_0)=0$ as well, then $w\equiv 0$ by standard ODE theory.
 Since $w$ is nontrivial,  we have $w'(y_0)\ne 0$, and then $\lim_{y\to y_0^+}U_{\theta}=+\infty$ or $-\infty$. 
  We claim that we must have $\lim_{y\to y_0^+}U_{\theta}=+\infty$. 
  Indeed, if $\lim_{y\to y_0^+}U_{\theta}=-\infty$, then there exist a sequence $y_k\to y_0^+$, such that $U'_{\theta}(y_k)>0$ and $\lim_{k\to \infty}U_{\theta}(y_k)=-\infty$. Then by (\ref{eq:NSE}), we have 
  \[
  \begin{split}
      +\infty>P_c(y_0) & =\lim_{k\to \infty}P_c(y_k) = \lim_{k\to \infty}(\nu (1-y_k^2)U'_{\theta}(y_k)+2\nu y_kU_{\theta}(y_k)
      +\frac{1}{2}U_{\theta}^2(y_k))\\
      & \ge \lim_{k\to \infty}(2y_kU_{\theta}(y_k)+\frac{1}{2}U_{\theta}^2(y_k))=+\infty, 
    \end{split}
  \]
  which is a contradiction. So we have $\lim_{y\to y_0^+}U_{\theta}=+\infty$. 
  Similarly, if $y_1<1$, then  $\lim_{y\to y_1^-}U_{\theta}=-\infty$ and (b) holds. The proof is finished. 
\end{proof}

\begin{rmk}
  Lemma \ref{lem3_3} in particular implies that if $c_1<-\nu^2$, then there is no solution of (\ref{eq:NSE}) on $(-1, -1+\delta)$ for any $\delta>0$. Otherwise Lemma \ref{lem3_3} (a) would imply $U_{\theta}(-1)=-2\nu\pm 2\sqrt{c_1+\nu^2}$, which is not a real number, leading to a contradiction. 
  This implies that $U_{\theta}=U_{\theta}[w]$ for some $w$ with infinite zero points $y_k\in (-1, 1)$ satisfying $\lim_{k\to \infty}y_k=-1$. 
  Similarly, if $c_2<-\nu^2$, then there does not exist a solution of (\ref{eq:NSE}) on $(1-\delta, 1)$ for any $\delta>0$. 
\end{rmk}

By Theorem \ref{thmLLY2_0}, Proposition \ref{prop2_1}, Lemma \ref{lem3_1} and Lemma \ref{lem3_3}, if $c\in J_{\nu}$, then the local solutions of (\ref{eq:NSE}) satisfy the following properties.

\begin{prop}\label{prop3_1}
  Let $c\in J_{\nu}$, 
  and  $U_{\theta}$ be a solution of (\ref{eq:NSE}) whose domain in $(-1, 1)$ is $(y_0, y_1)$ for some $-1\le y_0<y_1\le 1$. 
  Then one of the following holds.\\
  (a) $y_0=-1$, $y_1=1$, $U_{\theta}=U^{c, \gamma}_{\theta}$ for some $\gamma^-(c)\le \gamma\le \gamma^+(c)$. In particular, $\lim_{y\to -1^+} U_{\theta}$ and $\lim_{y\to 1^-} U_{\theta}$ both exist and are finite, satisfying (\ref{eqLLY2_1}).\\
  (b) $-1=y_0<y_1<1$, $U_{\theta}<U_{\theta}^-$, $U_{\theta}(-1)=\lim_{y\to -1^+} U_{\theta}=\tau_1$ and $\lim_{y\to y_1^-} U_{\theta}=-\infty$. \\
  (c) $-1<y_0<y_1=1$, $U_{\theta}>U_{\theta}^+$, $U_{\theta}(1)=\lim_{y\to 1^-} U_{\theta}=\tau_2'$ and $\lim_{y\to y_0^+}U_{\theta}=+\infty$. 
\end{prop}
\begin{proof}
  Let $U_{\theta}$ be a solution of (\ref{eq:NSE}) in $(y_0, y_1)$, we first show that either $y_0=-1$ or $y_1=1$. Since $c\in J_{\nu}$, by Theorem \ref{thmLLY2_0}, there exists a global solution $\bar{U}_{\theta}\in C^{\infty}(-1, 1)$ of (\ref{eq:NSE}). If $-1<y_0<y_1<1$, then by Lemma \ref{lem3_3} (b), we must have 
  $\lim_{y\to y_0^+}U_{\theta}=+\infty$ and $\lim_{y\to y_1^-}U_{\theta}=-\infty$.
  Since $\bar{U}_{\theta}\in C([y_0, y_1])$ is bounded, there exists some $\tilde{y}\in (y_0, y_1)$, such that $U_{\theta}(\tilde{y})=\bar{U}_{\theta}(\tilde{y})$. By standard ODE theory, we must have $U_{\theta}\equiv \bar{U}_{\theta}$, and thus $y_0=-1$ and $y_1=1$, contradiction. 
  So we have either $y_0=-1$ or $y_1=1$. 
    
  If $y_0=-1$ and $y_1=1$, then $U_{\theta}$ is a global solution of (\ref{eq:NSE}) on $(-1, 1)$. By Theorem \ref{thmLLY2_0}, there exists some $\gamma\in [\gamma^-(c), \gamma^+(c)]$, such that $U_{\theta}=U^{c,\gamma}_{\theta}$. So it satisfies all properties in Proposition \ref{prop2_1}. In particular, $\lim_{y\to -1^+} U_{\theta}$ and $\lim_{y\to 1^-} U_{\theta}$ both exist and are finite, and are given by (\ref{eqLLY2_1}). So (a) holds in this case.
    
  If $y_0=-1$ and $y_1<1$, then by Lemma \ref{lem3_3} (a), we have $U_{\theta}(-1)=\lim_{y\to -1^+} U_{\theta}=\tau_1$ or $\tau_2$. By Proposition \ref{prop2_1} (ii), for any $\delta>0$, there is only one solution $U^+_{\theta}$ of (\ref{eq:NSE}) on $(-1, -1+\delta)$ satisfying $U^+_{\theta}(-1)=\tau_2$. Since $y_1<1$, we have $U_{\theta}\ne U^+_{\theta}$ and thus $U_{\theta}(-1)=\tau_1$. By Lemma \ref{lem3_3} (b), since $y_1<1$, we have $\lim_{y\to y_1^-}U_{\theta}=-\infty$. 
    
  We also claim that $U_{\theta}<U^-_{\theta}$ on $(-1, y_1)$. If not, there is some $-1<\tilde{y}<y_1$, such that $U_{\theta}^-(\tilde{y})<U_{\theta}(\tilde{y})$. Since $U_{\theta}^-\in L^{\infty}([-1, 1])$ and $\lim_{y\to y_1^-}U_{\theta}=-\infty$, there exists some $\tilde{y}<\bar{y}<y_1$, such that $U_{\theta}(\bar{y})=U_{\theta}^-(\bar{y})$. By standard ODE theory, we must have $U_{\theta}\equiv U^-_{\theta}$ and is a global solution on $(-1, 1)$. Contradiction. So we have $U_{\theta}<U^-_{\theta}$ on $(-1, y_1)$, and (b) holds when $y_0=-1$ and $y_1<1$. 

  If $y_0>-1$ and $y_1=1$, then by similar arguments, we have (c) holds. The proof is finished. 
\end{proof}

We may also establish a 1-1 correspondence between all local but not global solutions of (\ref{eq:NSE}) and their blow-up point. 

\begin{prop}
  Let $\nu>0$, $c=(c_1, c_2, c_3)\in \R^3$. \\
  (i) For each $y_0\in (-1, 1)$, there exist some $\delta>0$, depending only on $c, \nu$ and $y_0$, and a unique solution $U_{\theta}$ of (\ref{eq:NSE}) on $(y_0, y_0+\delta)$, such that $\lim_{y\to y_0^+}U_{\theta}=+\infty$. Moreover, these are all the solutions of (\ref{eq:NSE}) on  $(y_0, y_0+\delta)\subset (-1, 1)$ such that $\lim_{y\to y_0^+}U_{\theta}$ does not exist. \\
  (ii) For each $y_1\in (-1, 1)$, there exist some $\delta>0$, depending only on $c, \nu$ and $y_1$, and a unique $U_{\theta}$ of (\ref{eq:NSE}) on $(y_1-\delta, y_1)$, such that $\lim_{y\to y_1^-}U_{\theta}=-\infty$. Moreover, these are all the  solutions of (\ref{eq:NSE}) on $(y_1-\delta, y_1)\subset (-1, 1)$ such that $\lim_{y\to y_1^-}U_{\theta}$ does not exist.
\end{prop}
\begin{proof}
  We only prove (i). The proof of (ii) is similar. Let $y_0\in (-1, 1)$ be fixed and $C_0\ne 0$ be a constant. 
  By standard ODE theory, there exists a unique solution $w$ of (\ref{eqH_0_2}) on $(-1, 1)$ satisfying $w(y_0)=0$ and $w'(y_0)=C_0$. 
  Since $w'(y_0)\ne 0$, there exists some $\delta>0$, such that $w\ne 0$ on $(y_0, y_0+\delta)$. 
  By direct computation, 
  $U_{\theta}=U_{\theta}[w]$ defined by (\ref{eqH_0_1}) is a solution of (\ref{eq:NSE}) on $(y_0, y_0+\delta)$. By Lemma \ref{lem3_3}, we have $\lim_{y\to y_0^+}U_{\theta}=+\infty$. 

  Next, we prove the uniqueness of the solution of (\ref{eq:NSE}) in $(y_0, y_0+\delta)$ such that $\lim_{y\to y_0^+}U_{\theta}$ does not exist. 
  Assume $U_{\theta}$ is such a solution. 
  By Lemma \ref{lem3_1} and Lemma \ref{lem3_3}, 
   $U_{\theta}$ is given by (\ref{eqH_0_1}) with some solution $w\in C^{\infty}(-1, 1)$ of (\ref{eqH_0_2}) satisfying $w\ne 0$ on $(y_0, y_0+\delta)$,  
   and  $\lim_{y\to y_0^+}U_{\theta}=+\infty$. 
  This implies $w(y_0)=0$ and $w'(y_0)=C_0$ for some constant $C_0\ne 0$. We show $U_{\theta}$ is uniquely determined by $y_0$ and is independent of $C_0$. To see this, note that $U_{\theta}$ depends only on $w'/w$. Let $w_1, w_2$ be the two fundamental solutions of (\ref{eqH_0_2}), then $w=C_1w_1+C_2w_2$ for some constants $C_1, C_2$ satisfying 
  \begin{equation}\label{eq3_3_1}
    w(y_0)=C_1w_1(y_0)+C_2w_2(y_0)=0, \quad w'(y_0)=C_1w'_1(y_0)+C_2w'_2(y_0)=C_0.
  \end{equation}
  First, $w_1(y_0)$ and $w_2(y_0)$ cannot both be zero, otherwise 
  $w_1(y_0)w_2'(y_0)-w_1'(y_0)w_2(y_0)=0$, which is a contradiction since $w_1, w_2$ are linearly independent fundamental solutions of (\ref{eqH_0_2}). 
  If $w_1(y_0)=0$, then we have $w_2(y_0)\ne 0$, and by the first equation in (\ref{eq3_3_1}), we have $C_2=0$. 
  Then $w=C_1w_1$ 
  and $U_{\theta}=U_{\theta}[w_1]$. 

  If $w_1(y_0)\ne 0$, then we must have $C_2\ne 0$. Otherwise the first equation in (\ref{eq3_3_1}) implies $C_1=0$ as well and then $w\equiv 0$, contradiction. 
  So we have $D:=\frac{C_1}{C_2}=-\frac{w_2(y_0)}{w_1(y_0)}$ is well-defined and 
  $U_{\theta}=U_{\theta}[w]$ is given by (\ref{eqH_0_1}) with $w=Dw_1+w_2$. Note $D$ is uniquely determined by $y_0$, we have  
   $U_{\theta}$ is uniquely determined by the prescribed blow-up point $y_0$. The proof is finished.
\end{proof}

In general, the expressions of the fundamental solutions of (\ref{eqH_0_2}) are not known. However, when $c_1, c_2\ge -\nu^2$, we are able to convert (\ref{eq:NSE}) into some hypergeometric differential equation, whose fundamental solutions are well-known. 

Let $\nu>0$ and $c=(c_1, c_2, c_3)\in\R^3$ satisfying $c_1, c_2\ge -\nu^2$. Denote 
\begin{equation}\label{eq_a}
  a(c,\nu)=-\sqrt{c_1+\nu^2}+\nu, \quad b(c, \nu)= \sqrt{c_2+\nu^2}-\nu, \quad \lambda(c, \nu)=\frac{c_1+c_2}{2}+c_3-ab, 
\end{equation}
\be\label{eq_A}
  A(c, \nu)=\frac{a-b-\nu+ \sqrt{(a-b-\nu)^2-2\lambda}}{2\nu}, B(c, \nu)=\frac{a-b-\nu- \sqrt{(a-b-\nu)^2-2\lambda}}{2\nu}, 
\ee
and $C(c, \nu)=a/\nu$.
\begin{lem}\label{lem3_2}
  Let $\nu>0$, $ c=(c_1, c_2, c_3)\in\R^3$, $c_1, c_2\ge -\nu^2$, and $-1\le y_0< y_1\le 1$. Let $a, b, \lambda$ be defined by (\ref{eq_a}) and $A, B, C$ be defined by (\ref{eq_A}). 
  Then $U_{\theta}\in C^{1}(y_0, y_1)$ is a solution of (\ref{eq:NSE}) on $(y_0, y_1)$ if and only if 
  \begin{equation}\label{eq2_1_1}
    U_{\theta}
    =a(1-y)+b(1+y)+\nu(1-y^2)\frac{\xi'(\frac{1+y}{2})}{\xi(\frac{1+y}{2})}
  \end{equation}
  for some $\xi(z)\in C^{2}(\frac{1+y_0}{2}, \frac{1+y_1}{2})$, 
  where $'=\frac{d}{dz}$, 
  such that $\xi\ne 0$ on  $(\frac{1+y_0}{2}, \frac{1+y_1}{2})$ and 
  \begin{equation}\label{eq2_1_2}
    z(1-z)\frac{d^2\xi}{dz^2}+(C-(A+B+1)z)\frac{d\xi}{dz}-AB \xi=0, \quad 
    \frac{1+y_0}{2}<z<\frac{1+y_1}{2}.
  \end{equation}
\end{lem}
\begin{proof}
  Let $\nu>0$, 
  $c_1, c_2\ge -\nu^2$, and $a, b, \lambda$ be defined by (\ref{eq_a}). Define $c_3'=ab-\frac{c_1+c_2}{2}$ and $U_{\theta,0} = a (1-y) + b (1+y)$. Then $U_{\theta, 0}$ satisfies 
  \be\label{eq2_1_0}
    \nu (1-y^2) U_{\theta, 0}' + 2 \nu y U_{\theta, 0} + \frac{1}{2} U_{\theta, 0}^2 = c_1(1-y) + c_2 (1+y) +c_3' (1-y^2), \quad -1<y<1.
  \ee
  Let $g = U_\theta - U_{\theta,0}$. 
  Note $\lambda=c_3-c_3'$. 
  A direct computation shows that $U_{\theta}$ is a solution of (\ref{eq:NSE}) in $(y_0, y_1)\subset (-1, 1)$ if and only if $g$ is a solution of 
  \begin{equation}\label{eq2_1_3}
    \nu(1-y^2)g' + \Big( (-a+b+2\nu)y + a + b \Big) g + \frac{1}{2}g^2 = \lambda(1-y^2), \quad y_0<y<y_1.
  \end{equation}
  We show $g\in C^1(y_0, y_1)$ is a solution of  the above if and only if $g=\nu(1-y^2)\frac{\xi'(\frac{1+y}{2})}{\xi(\frac{1+y}{2})}$ for some $\xi(z)\in C^{2}(\frac{1+y_0}{2}, \frac{1+y_1}{2})$ satisfying (\ref{eq2_1_2}) and $\xi\ne 0$ on  $(\frac{1+y_0}{2}, \frac{1+y_1}{2})$. 
  
 Assume that $g\in C^1(y_0, y_1)$ satisfies (\ref{eq2_1_3}). Let $\xi(z):=e^{\int_{y_0}^{2z-1}\frac{g(s)}{2\nu (1-s^2)}}ds$. Then $\xi(z)\in C^{2}(\frac{1+y_0}{2}, \frac{1+y_1}{2})$, $\xi\ne 0$ on  $(\frac{1+y_0}{2}, \frac{1+y_1}{2})$, and $g(y)=\nu(1-y^2)\frac{\xi'(\frac{1+y}{2})}{\xi(\frac{1+y}{2})}$,  where the derivative $'=\frac{d}{dz}$. 
  By (\ref{eq2_1_3}), we have 
  \be\label{eq2_1_4}
    z(1-z)\frac{d^2\xi}{d z^2} + \frac{1}{\nu}\Big(a+(b-a)z \Big) \frac{d\xi}{d z} - \frac{\lambda}{2\nu^2} \xi = 0, \quad \frac{1+y_0}{2}<z<\frac{1+y_1}{2}. 
  \ee
  Let $A, B, C$ be defined by (\ref{eq_A}). It can be verified that (\ref{eq2_1_4}) is equivalent to (\ref{eq2_1_2}). 

  On the other hand, assume $\xi(z)$ is a $C^2$ solution of (\ref{eq2_1_2}) satisfying $\xi(z)\ne 0$ on $(\frac{1+y_0}{2}, \frac{1+y_1}{2})$. Then by the definitions of $A, B, C$, we have (\ref{eq2_1_4}). 
  Let $y=2z-1$ and $g=\nu(1-y^2)\frac{\xi'(\frac{1+y}{2})}{\xi(\frac{1+y}{2})}$, 
  then  $g$ satisfies (\ref{eq2_1_3}). 
  Let $U_{\theta}=U_{\theta, 0}+g$. 
  By (\ref{eq2_1_3}) and (\ref{eq2_1_0}), we have $U_{\theta}$ satisfies (\ref{eq:NSE}). 
  The proof is finished. 
\end{proof}

\begin{rmk}
  Note in Lemma \ref{lem3_2}, the constants $a=\tau_1/2$ and $b=\tau_2'/2$. We could also choose $a=\tau_2/2$ or $b=\tau_1'/2$, where $\tau_1, \tau_2, \tau_1', \tau_2'$ are defined in (\ref{eq_2}). 
\end{rmk}

By Lemma \ref{lem3_2}, we have converted (\ref{eq:NSE}) into (\ref{eq2_1_2}). This equation is in the form of the Euler's hypergeometric differential equation, whose fundamental solutions have been well-studied (see, e.g. \cite{OLBC2010,EMOT1953} for general theories about hypergeometric equations). 
For instance, if $C$ is not an integer, then the two fundamental solutions of (\ref{eq2_1_2}) are 
\begin{equation*}
  \xi_1(z)={}_2 F_1 (A,B,C, z), \quad \xi_2(z)=z^{1-C}{}_2 F_1 (1+A-C, 1+B-C, 2-C, z).
\end{equation*}
Here ${}_2 F_1 (A,B,C, z)$ is the hypergeometric series 
${}_2 F_1 (A,B,C, z)=\sum_{n=0}^{\infty}\frac{(A)_n(B)_n}{(C)_c}\frac{z^n}{n!}$. 
where $(q)_n$ denote the rising Pochhanmmer symbol for any $q\in \bC$. We will make use of the representation of $U_{\theta}$ with the help of hypergeometric series to sketch graphs of some examples in the next subsection.

Similar as Lemma \ref{lem3_1}, we have the following expression of the solution of (\ref{eqN_1}) when $c_1, c_2\ge -\nu^2$. 
Let $a, b, A, B, C$ be defined by (\ref{eq_a})-(\ref{eq_A}), and $\xi_1, \xi_2$ be the fundamental solutions of (\ref{eq2_1_2}). 
\begin{cor}\label{cor3_1}
Let $\nu>0$, $c=(c_1, c_2, c_3)\in \R^3$, $c_1, c_2\ge -\nu^2$, $\bar{y}\in (-1, 1)$ and $\gamma\in \R$. 
Let $\tilde{\mu}=\frac{\gamma-a(1-\bar{y})-b(1+\bar{y})}{\nu (1-\bar{y}^2)}$ and $\bar{z}=\frac{1+\bar{y}}{2}$. When $\xi_1'(\bar{z})\ne \tilde{\mu} \xi_1(\bar{z})$, set $\tilde{D}:=-\frac{\xi_2'(\bar{z})-\tilde{\mu} \xi_2(\bar{z})}{\xi_1'(\bar{z})-\tilde{\mu} \xi_1(\bar{z})}$. 
Then there exist some $y_0, y_1$ satisfying $-1\le y_0<\bar{y}<y_1\le 1$, and a unique solution $U_{\theta}\in C^2(y_0, y_1)$ of (\ref{eqN_1}), given by (\ref{eq2_1_1}) with 
   \be
      \xi(z)=\left\{
        \begin{array}{ll}
        \tilde{D}\xi_1(z)+\xi_2(z), & \textrm{ if }\xi_1'(\frac{1+\bar{y}}{2})\ne \tilde{\mu} \xi_1(\frac{1+\bar{y}}{2}),\\
        \xi_1(z),  &\textrm{ if } \xi_1'(\frac{1+\bar{y}}{2})= \tilde{\mu} \xi_1(\frac{1+\bar{y}}{2}). 
        \end{array}
      \right.
   \ee
\end{cor}
 
\subsection{Graphs of example solutions}
 
In this subsection, we present several graphs illustrating examples of $(-1)$-homogeneous axisymmetric no-swirl solutions of (\ref{NS}).
Let $J_{\nu}$ be defined by (\ref{eq_J}) and $\gamma^{\pm}(c)$ be defined as in Theorem \ref{thmLLY2_0}. According to the classification of global solutions in Theorem \ref{thmLLY2_0}, we consider the following cases. 
\begin{equation}\label{eq_case}
  \begin{array}{ll}
    \textrm{Case 1. } c_1, c_2>-\nu^2, c_3>\bar{c}_3; & \quad \textrm{Case 2. }c_1=-\nu^2, c_2>-\nu^2, c_3>\bar{c}_3; \\
    \textrm{Case 3. }c_1>-\nu^2, c_2=-\nu^2, c_3>\bar{c}_3; & \quad \textrm{Case 4. }c_1=c_2=-\nu^2, c_3>\bar{c}_3;\\
    \textrm{Case 5. } c_1, c_2\ge -\nu^2, c_3=\bar{c}_3; & \quad  \textrm{Case 6. }c\notin J_{\nu}. 
  \end{array}
\end{equation}
By Theorem \ref{thmLLY2_0}, if $c\in J_{\nu}$ and 
$\gamma^{-}(c)\le \gamma\le \gamma^+(c)$, 
then there exist a unique solution $U_{\theta}^{c, \gamma}(y)$ of (\ref{eq:NSE}) in $(-1, 1)$ satisfying $U_{\theta}(0)=\gamma$. 
If $c\in J_{\nu}$ and $\gamma\notin [\gamma^-(c), \gamma^+(c)]$, then there exists a solution of (\ref{eq:NSE}) satisfying $U_{\theta}^{c, \gamma}(0)=\gamma$ in a neighborhood of $y=0$, which blow up before $y$ reaches $-1$ or $1$. 
On the other hand, if $c\notin J_{\nu}$, then (\ref{eq:NSE}) admits only local solutions, and no global solution exists in this case. 
Using the representation of solutions in Lemma \ref{lem3_2} and Corollary \ref{cor3_1}, we present graphical illustrations of solution profiles for each of the six cases listed above. 

\medskip

\noindent\textbf{a. Structure of global and local solutions}
 
By Theorem \ref{thmLLY2_0}, for each fixed $c\in J_{\nu}$, the global solutions of (\ref{eq:NSE}) on $(-1, 1)$ are given by a one-parameter family $U^{c, \gamma}$. In Case 1-4, where $c_3>\bar{c}_3$, there is an upper solution $U^{+}_{\theta}(y)$ and a lower solution $U^-_{\theta}(y)$ in $C([-1, 1])\cap C^2(-1, 1)$. All other global solutions $U^{c, \gamma}(y)$ with $\gamma^-<\gamma<\gamma^+$ foliate the region between $U_{\theta}^{\pm}$, satisfying $U^{-}_{\theta}(y)<U^{c, \gamma}(y)<U^{+}_{\theta}(y)$ for $(-1, 1)$. 
Moreover, the boundary values $U_{\theta}(\pm 1)$ exist and are given by (\ref{eqLLY2_1}). In Case 5, where $c_3=\bar{c}_3$, there is only one global solution $U^*_{\theta}=\sin\theta u^*_{\theta}$ given through (\ref{eqLLY2_0_1}). In Case 6, where $c\notin J_{\nu}$, there is no global solution. 

Except $\{U^{c, \gamma}\mid (c, \gamma)\in I_{\nu}\}$, all other solutions of (\ref{eq:NSE}) are local solutions existing on a proper subinterval of $(-1, 1)$. There are three types of local solutions:
\medskip

\noindent \textbf{(a1)} $U_{\theta}\in C[-1, y_1)$ for some $y_1\in (-1, 1)$, $U_{\theta}(-1)=0$, $U_{\theta}<U^-_{\theta}$, $\lim_{y\to y_1^-}U_{\theta}=-\infty$;\\
\noindent\textbf{(a2)} $U_{\theta}\in C(y_0, 1]$ for some $y_0\in (-1, 1)$, $U_{\theta}(1)=0$, $U_{\theta}>U^+_{\theta}$, $\lim_{y\to y_0^+}U_{\theta}=+\infty$; \\
\noindent\textbf{(a3)} $U_{\theta}\in C(y_0, y_1)$ for $y_0, y_1\in (-1, 1)$, 
$\lim_{y\to y_0^+}U_{\theta}=+\infty$ and $\lim_{y\to y_1^-}U_{\theta}=-\infty$.

\medskip
If $c\in J_{\nu}$, then by Proposition \ref{prop3_1}, any non-global solution must satisfy either condition (a1) or (a2). If $c\notin J_{\nu}$, then solutions satisfying (a3) may occur. 

In the following, we provide some examples of solutions for $c\in \R^3$ in Case 1-6, and sketch the graphs and streamlines of some solutions. 

\medskip
 
\noindent \textbf{Case 1.} $c_1, c_2>-\nu^2$, $c_3>\bar{c}_3$. 
In this case, we have 
\be\label{eqcase1_1}
  U^{c, \gamma}(-1)=U_{\theta}^-(-1)<U^{+}_{\theta}(-1), \quad U^{c, \gamma}(1)=U_{\theta}^+(1)>U^{-}_{\theta}(1), \quad \gamma^-<\gamma<\gamma^+. 
\ee
All local solutions in this case satisfy either (a1) or (a2). 
\begin{example}\label{ex1}
  $\nu=1, c_1=c_2=0, c_3=0.5>\bar{c}_3$. 

  In \cite{LLY2024} (see example 3.2 there), we have obtained explicit but not elementary expressions for the solutions in this case by solving the corresponding hypergeometric equation (\ref{eq2_1_2}), where 
  \[
    u_\theta = \frac{ \sin\theta \big( K( \cos^2\frac{\theta}{2} ) - \alpha K( \sin^2\frac{\theta}{2} ) \big) }{2 \big( E( \cos^2\frac{\theta}{2} ) - \sin^2\frac{\theta}{2} K( \cos^2\frac{\theta}{2} ) + \alpha E( \sin^2\frac{\theta}{2} ) - \alpha \cos^2\frac{\theta}{2} K( \sin^2\frac{\theta}{2} ) \big) },
  \]
  for some constant $\alpha\in R$ satisfying $\gamma=\frac{K(1/2)(1-\alpha)}{(2E(1/2)-K(1/2))(1+\alpha)}$, and $K(x)$ and $E(x)$ are respectively the complete elliptic integrals of the first and second kind
  \[
    K(x) = \int_0^{\pi/2} \frac{1}{\sqrt{1-x \sin^2\theta} } d \theta, \qquad 
    E(x) = \int_0^{\pi/2} \sqrt{1-x \sin^2\theta d\theta}, \qquad 0<x<1. 
  \]
  If $0\le \alpha\le \infty$, then the solution is in $C^2(\mathbb{S}^2\setminus\{S, N\})$. 
  If $0<\alpha<\infty$, then $u=u^{c, \gamma}$ with $\gamma^-<\gamma<\gamma^+$, which is of Type 2 behavior as $x\to S$ or $N$ on $\bS^2$. 
  If $\alpha=0$, then $u=u^+(c)$, which is is of Type 3 behavior as $x\to S$ and Type 2 behavior as $x\to N$. If $\alpha=+\infty$, then $u=u^-(c)$, which is is of Type 2 behavior as $x\to S$ and Type 3 behavior as $x\to N$. 
  If $\alpha<0$, then the above solution is a local solution near $N$ or $S$ on $\mathbb{S}^2$. 
\end{example}
We present in Figure \ref{fig:case1} the graphs of the upper and lower solutions $U^{\pm}_{\theta}(y)$, some 
global solutions $U^{c, \gamma}_{\theta}(y)$ for $(c, \gamma)\in I_{\nu}$, and several local solutions with $c\in J_{\nu}$ and $\gamma\notin [\gamma^-, \gamma^+]$. All graphs are plotted over the interval $y\in [-1, 1]$. 
In particular, $U^{c, \gamma}(\pm 1)$ satisfy (\ref{eqcase1_1}), where 
$U_{\theta}^+(-1)=4$, $U_{\theta}^+(1)=0$, $U_{\theta}^-(-1)=0$, $U_{\theta}^-(1)=-4$. For $\gamma\in (\gamma^-, \gamma^+)$, we have $U_{\theta}^{c, \gamma}(-1)=U_{\theta}^{c, \gamma}(1)=0$. Moreover, $U_{\theta}^{\pm}$ are both Type 3 solutions. 
All $U_{\theta}^{c, \gamma}$ with $\gamma\in (\gamma^-, \gamma^+)$ are Type 2 solutions. 

In Figure \ref{fig:case1:1}, we plot the graphs of $u^{c, \gamma}_{\theta}, u^{c, \gamma}_{r}, U^{c, \gamma}_{\theta}, U^{c, \gamma}_{r}$ in the variable $\theta$ over the interval $(0, \pi)$. We also plot the streamlines of the corresponding velocity field in the meridional plane $(x_1, x_3)$, for various values of $\gamma$. These $2$D streamlines can be rotated about the $x_3$-axis to visualize the full $3$D axisymmetric flow. In particular, the last two plots in Figure \ref{fig:case1:1} illustrate local solutions that blow up at some angle $\theta=\theta_0$ as shown in the pictures, before the flow reaching either pole.

\begin{figure}[!htb]
	\centering
	\includegraphics[height=4cm]{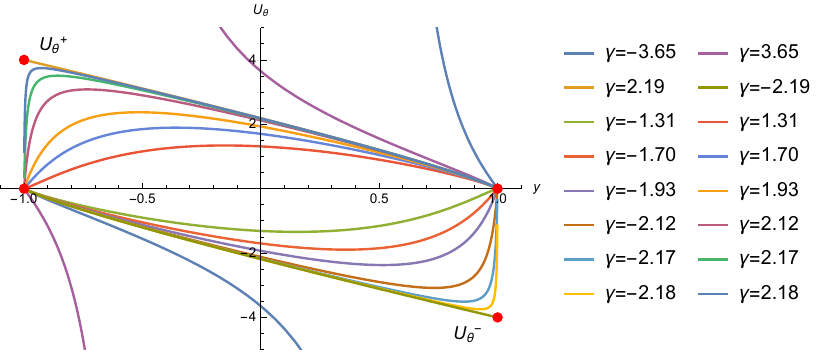}
	\caption{\small{The graphs of $U_{\theta}(y)$ in Example \ref{ex1} ($c_1=c_2=0, c_3=0.5$, Case 1).}}
	\label{fig:case1}
\end{figure}
\begin{figure}[!htb]
	\centering
	\includegraphics[height=3.75cm]{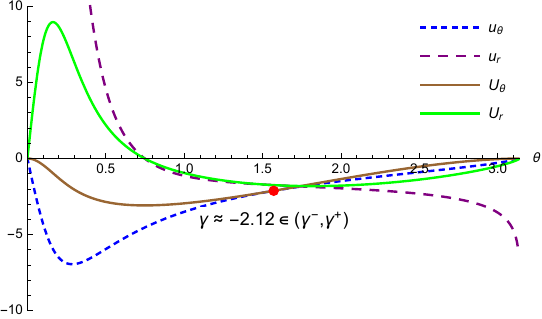}\hspace{1cm}
	\includegraphics[height=4cm]{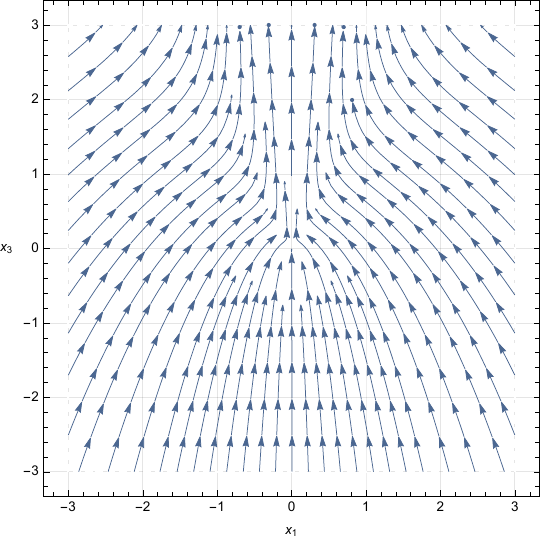}\\
	\includegraphics[height=3.75cm]{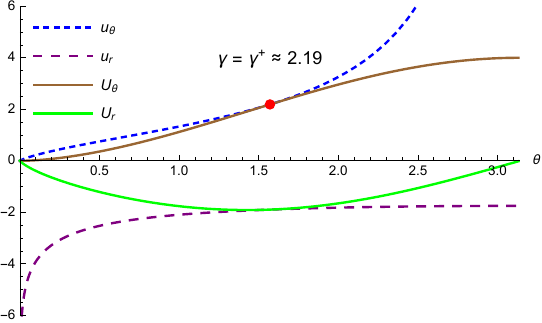}\hspace{1cm}
	\includegraphics[height=4cm]{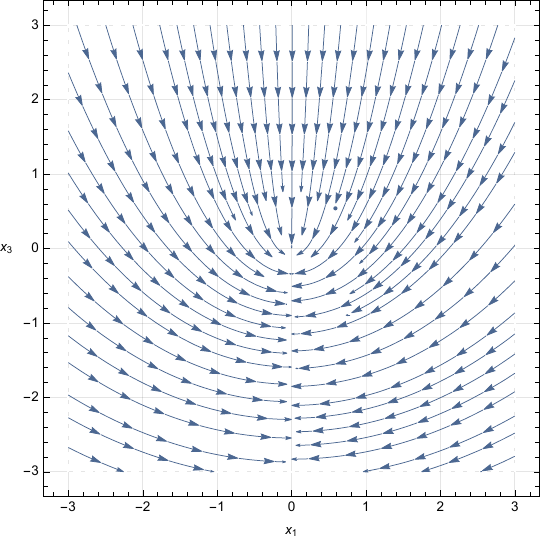}\\
	\includegraphics[height=3.75cm]{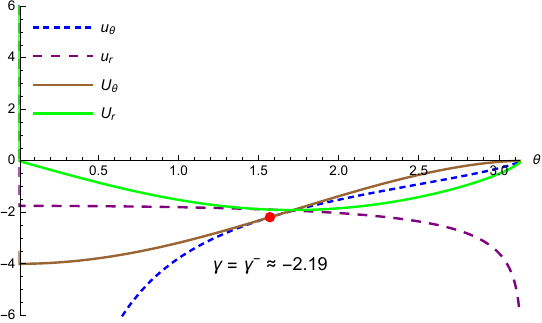}\hspace{1cm}
	\includegraphics[height=4cm]{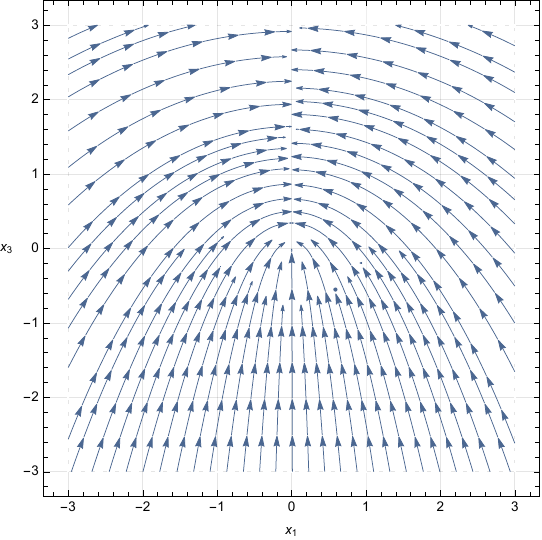}\\
	\includegraphics[height=3.75cm]{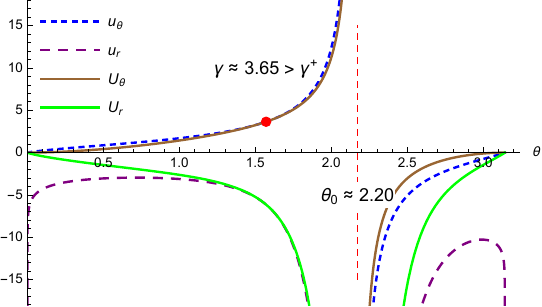}\hspace{1cm}
	\includegraphics[height=4cm]{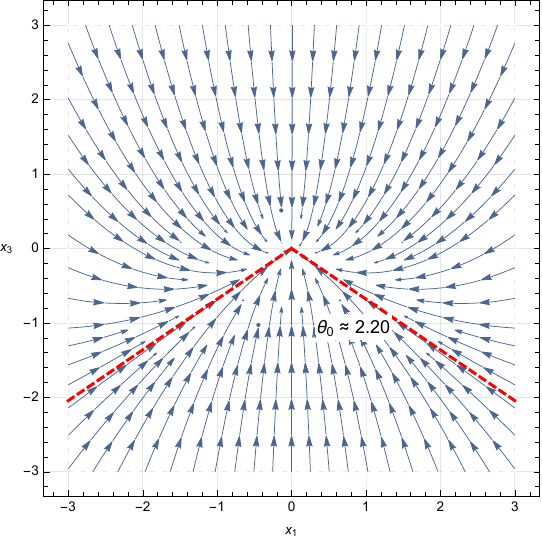}\\
	\includegraphics[height=3.75cm]{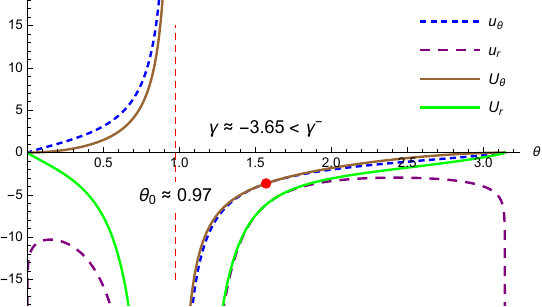}\hspace{1cm}
	\includegraphics[height=4cm]{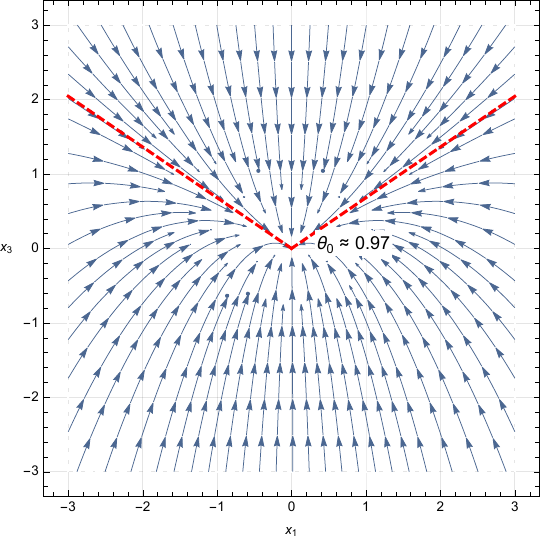}
	\caption{\small{The graphs and streamline of $u^{c, \gamma}$ in Example \ref{ex1} ($c_1=c_2=0, c_3=0.5$, 
	Case 1).}}
	\label{fig:case1:1}
\end{figure}

\medskip

\noindent\textbf{Case 2.} $c_1=-\nu^2, c_2>-\nu^2$, $c_3>\bar{c}_3$. 
In this case, we have 
\be\label{eqcase2_1}
  U^{c, \gamma}(-1)=U_{\theta}^{\pm}(-1)=2\nu, \quad U^{c, \gamma}(1)=U_{\theta}^+(1)>U^{-}_{\theta}(1), \quad \gamma^-<\gamma<\gamma^+. 
\ee
All local solutions are satisfy either (a1) or (a2), and all global solutions are of Type 3. 

\begin{example}\label{ex2_1}
  $\nu=1$, $c_1=-1, c_2=8, c_3=-1.5>\bar{c}_3$. In Figure \ref{fig:case3} (a), we present  the graphs of $U^{\pm}_{\theta}$, some global solutions $U^{c, \gamma}_{\theta}$ for $(c, \gamma)\in I_{\nu}$, and some local solutions near $-1$ and $1$. 
  The boundary values of $U^{\pm}_{\theta}$ are $U_{\theta}^+(-1)=U_{\theta}^-(-1)=2$, $U_{\theta}^+(1)=4$, $U_{\theta}^-(1)=-8$. For $\gamma\in (\gamma^-, \gamma^+)$, we have $U_{\theta}^{c, \gamma}(-1)=2$ and $U_{\theta}^{c, \gamma}(1)=4$. 
\end{example}

\noindent\textbf{Case 3.} $c_1>-\nu^2, c_2=-\nu^2$, $c_3>\bar{c}_3$. 
In this case, we have 
\be\label{eqcase3_1}
  U^{c, \gamma}(-1)=U_{\theta}^{-}(-1)<U_{\theta}^{+}(-1), \quad U^{c, \gamma}(1)=U_{\theta}^{\pm}(1)=-2\nu, \quad \gamma^-<\gamma<\gamma^+. 
\ee
All local solutions satisfy either (a1) or (a2). All global solutions are of Type 3. 
 
\begin{example}\label{ex3_1}
$\nu=1$, $c_1=8, c_2=-1, c_3=-1.5>\bar{c}_3$. In Figure \ref{fig:case3} (b),  we present the graphs of $U^{\pm}_{\theta}$, some global solutions $U^{c, \gamma}_{\theta}$ for $(c, \gamma)\in I_{\nu}$, and some local solutions near $-1$ and $1$. 
The boundary values of $U^{\pm}_{\theta}$ are $U_{\theta}^+(-1)=8$, $U_{\theta}^-(-1)=-4$, $U_{\theta}^+(1)=U_{\theta}^-(1)=-2$. For $\gamma\in (\gamma^-, \gamma^+)$, we have $U_{\theta}^{c, \gamma}(-1)=-4$, $U_{\theta}^{c, \gamma}(1)=-2$. 
\end{example}

\begin{figure}[!htb]
	\centering
	(a)\includegraphics[height=3.75cm]{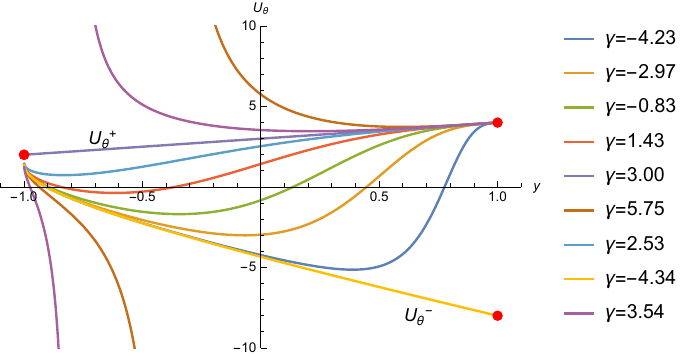}
	(b)\includegraphics[height=3.75cm]{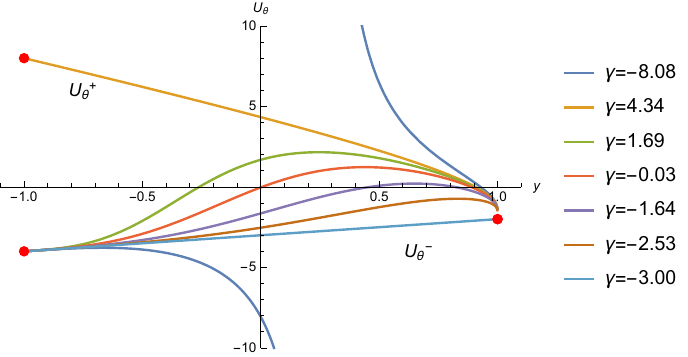}
	\caption{\small{(a) The graphs of $U_{\theta}(y)$ in Example \ref{ex2_1}; (b) The graphs of $U_{\theta}(y)$ Example \ref{ex3_1}}}
	\label{fig:case3}
\end{figure}

\medskip 

\noindent\textbf{Case 4.} $c_1=c_2=-\nu^2$, $c_3>\bar{c}_3$. 
In this case, $U^{c, \gamma}(-1)=2\nu$ and $U^{c, \gamma}(1)=-2\nu$ for all $\gamma^-\le \gamma\le \gamma^+$. 
All local solutions satisfy either (a1) or (a2). All global solutions are of Type 3. 

\begin{example}\label{ex4_1}
  $\nu=1$, $c_1=c_2=-1, c_3=0.5>\bar{c}_3$. We present in Figure \ref{fig:case4} the graphs of $U^{\pm}_{\theta}$, some global solutions $U^{c, \gamma}_{\theta}$ for $(c, \gamma)\in I_{\nu}$, and some local solutions near $-1$ and $1$. In particular, $U_{\theta}^{c, \gamma}(-1)=2$ and $U_{\theta}^{c, \gamma}(1)=-2$ for all $\gamma\in [\gamma^-, \gamma^+]$.
\end{example}

\begin{figure}[!htb]
	\centering
	\includegraphics[height=4cm]{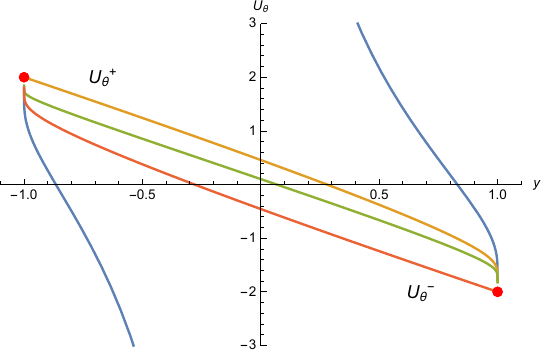}
	\caption{\small{The graphs of $U_{\theta}(y)$ in Example \ref{ex4_1} ($c_1=-1, c_2=-1, c_3=0.5$, Case 4)}}
	\label{fig:case4}
	\vspace{-0.5cm}
\end{figure}

\noindent\textbf{Case 5.} $c_3=\bar{c}_3(c_1,c_2,\nu)$.
 In this case, there is a unique global solution 
\begin{equation}\label{eq5_1}
  U_{\theta}^+(c)\equiv U^-_{\nu, \theta}(c)\equiv U^*_{\theta}(c_1,c_2):= (\nu+\sqrt{\nu^2+c_1})(1-y)+(-\nu-\sqrt{\nu^2+c_2})(1+y).
\end{equation}
This solution is of Type 3, satisfying 
$
  U_{\theta}^*(-1)=\tau_2=2\nu+2\sqrt{\nu^2+c_1}$ and $U_{\theta}^*(1)=\tau_1'=-2\nu-2\sqrt{\nu^2+c_2}. 
$
All other solutions $U_{\theta}$ in this case are local solutions with satisfying either (a1) or (a2). 
The values of $U_{\theta}(\pm 1)$ exhibit different behavior when $c_1, c_2>-\nu^2$ or $c_1=-\nu^2$ or $c_2=-\nu^2$. We provide an example in each case. 

\begin{example}\label{ex5_1}
  (a) $\nu=1$, $c_1=c_2=0, c_3=\bar{c}_3=-4$; \\
  (b) $\nu=1$, $c_1=-1, c_2=8, c_3=\bar{c}_3=-7.5$; \\
  (c) $c_1=8, c_2=-1, c_3=\bar{c}_3=-7.5$; \\
  (d) $c_1=c_2=-1, c_3=\bar{c}_3=0$. \\
  See Figure \ref{fig:case5:1} for the graphs of $U_{\theta}^{c, \gamma}(y)$ in (a)-(d). 
\end{example}
\begin{figure}[!htb]
\vspace{-0.5cm}
	\centering
	(a) \includegraphics[height=3.75cm]{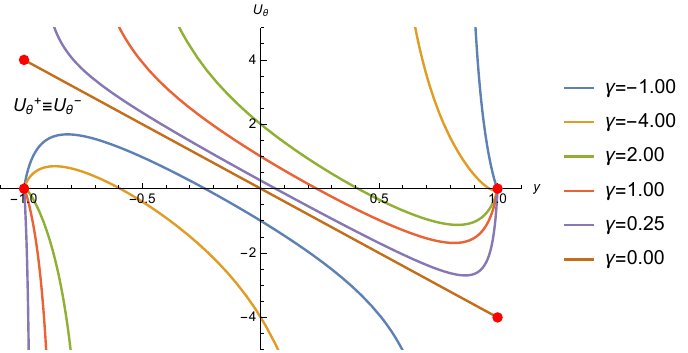}
	(b) \includegraphics[height=3.75cm]{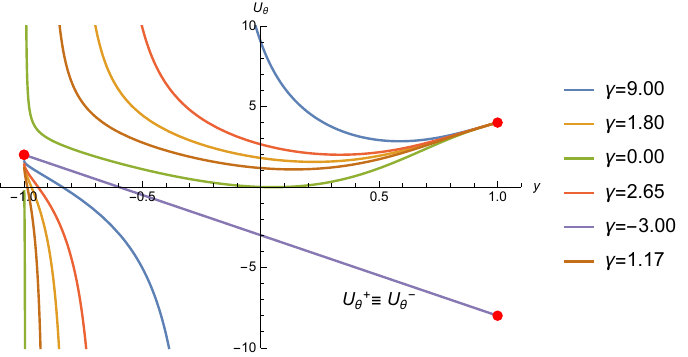}\\
	(c) \includegraphics[height=3.75cm]{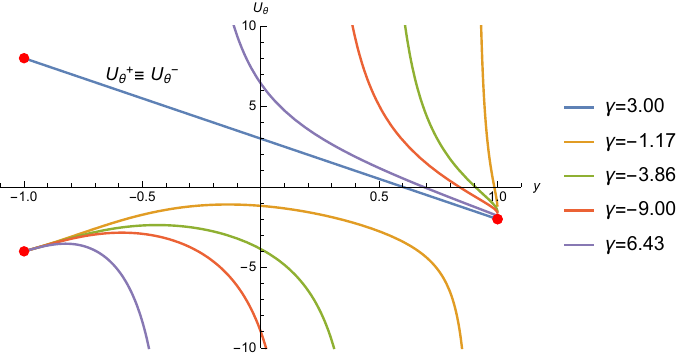}
	(d) \includegraphics[height=3.75cm]{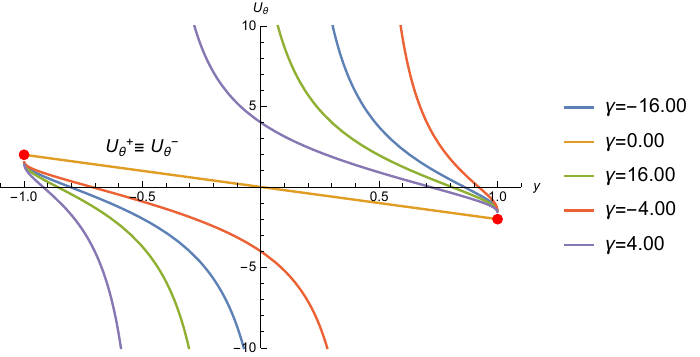}
	\caption{\small{The graphs of $U_{\theta}(y)$ in Example \ref{ex5_1} ($c_3=\bar{c}_3$, Case 5)}}
	\label{fig:case5:1}
\end{figure}

\noindent\textbf{Case 6.} $c\notin J_{\nu}$. In this case, no global solution exists. There are three types of local solutions satisfying (a1)-(a3) as described above. Here are some more details about their behaviors.  

\medskip

\noindent \textbf{(a1)} Solutions exist on $(-1, y_1)$ for some $y_1<1$ and $\lim_{y\to y_1^-}U_{\theta}=-\infty$. Among these solutions, there is an upper solution $U_{\theta}^+$, which has the largest domain and satisfies $U_{\theta}^+(-1)=\tau_2$. 
Any other (a1) solution satisfies $U_{\theta}<U^+_{\theta}$ in the intersection of their respective domains, and $U_{\theta}(-1)=\tau_1$. 

\noindent \textbf{(a2)} Solutions exist on on $(y_0, 1)$ for some $-1<y_0<1$ and $\lim_{y\to y_0^+}U_{\theta}=+\infty$. Among these solutions, there is an lower solution $U^-_{\theta}$, which has the largest domain and satisfies $U_{\theta}^-(1)=\tau_1'$. 
Any other (a2) solution satisfies $U_{\theta}>U^-_{\theta}$ in the intersection of their respective domains,  and $U_{\theta}(1)=\tau_2'$. 

\noindent \textbf{(a3)} Solutions exist on $(y_0, y_1)\subset\subset (-1, 1)$, with $\lim_{y\to y_0^+}U_{\theta}=+\infty$ and $\lim_{y\to y_1^-}U_{\theta}=-\infty$. The graphs of such solutions lie above all (a1) solutions and below all (a2) solutions within the intersection of their respective domains. 

\begin{example}\label{ex6_1}
  (a) $\nu=1$, $c_1=c_2=0$, $c_3=-12<\bar{c}_3$. \\
  (b) $\nu=1$, $c_1=-1, c_2=8$, $c_3=-17.5<\bar{c}_3$.\\
  (c) $\nu=1$, $c_1=8, c_2=-1$, $c_3=-17.5<\bar{c}_3$.\\
  (d) $\nu=1$, $c_1=c_2=-1, c_3=-1.5<\bar{c}_3$. \\
  See Figure \ref{fig:case6} for the graphs of $U_{\theta}$. 
\end{example}
\begin{figure}[!htb]
\vspace{-0.5cm}
	\centering
	(a) \includegraphics[height=3.5cm]{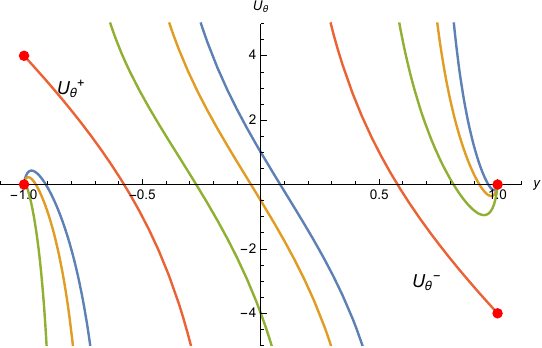}\hspace{1cm}
	(b) \includegraphics[height=3.5cm]{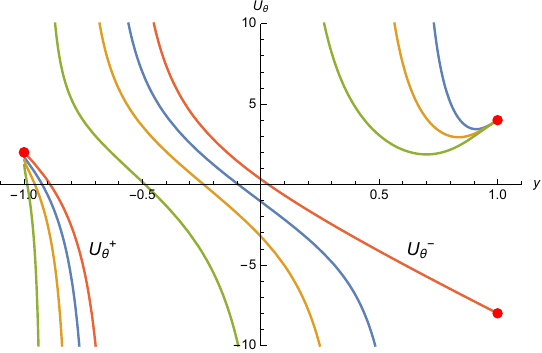}\\
	(c) \includegraphics[height=3.5cm]{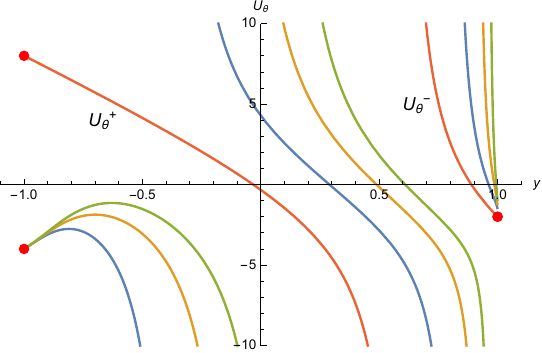}\hspace{1cm}
	(d)\includegraphics[height=3.5cm]{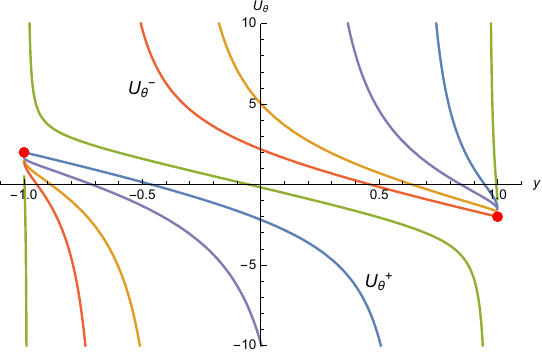}
	\caption{\small{The graphs of $U_{\theta}(y)$ in Example \ref{ex6_1} ($c_3<\bar{c}_3$, no global solution, Case 6)}}
	\label{fig:case6}
\end{figure}

\medskip

\noindent\textbf{b. Vanishing viscosity limit} 

In Section \ref{sec2_3}, we discussed the behavior of $(-1)$-homogeneous axisymmetric no-swirl solutions of (\ref{NS}) as $\nu\to 0$. In particular, it was shown in \cite{LLY2019b} that some sequences of solutions $(u_{\nu_k}, p_{\nu_k})$ of (\ref{NS}) converge in the $C^m_{loc}$ topology to $(v^{\pm}_c, q^{\pm}_c)$ on $\bS^2\setminus\{S, N\}$ as $\nu_k\to 0$, while for other sequences of solutions, transition layer behaviors occur. Specifically, for every latitude circle, there are sequences $(u_{\nu_k}, p_{\nu_k})$ which $C^m_{loc}$ converge respectively to different solutions of the Euler equations on the spherical caps above and below the latitude circle. Below we present an example exhibiting such transition layer behavior. 

\begin{example}\label{exb_1}
  Let $c_1=\frac{25}{9}, c_2=\frac{1}{9}, c_3=-2$. Then (\ref{eq:NSE}) reads as 
  \[
    \nu (1-y^2)U'_{\nu, \theta}+2\nu yU_{\nu, \theta}+\frac{1}{2}U^2_{\nu, \theta}=P_c(y)=2(y-\frac{2}{3})^2.
  \] 
  As $\nu\to 0$, the above equation tends to the Euler's equation for $(-1)$-homogeneous axisymmetric solutions
  $
    \frac{1}{2}V^2_{\theta}=2(y-\frac{2}{3})^2,
  $
  where $V_{\theta}=\sin\theta v_{\theta}$. 
  There are two solutions 
  $
    V_{\theta}^1=2(y-\frac{2}{3})$ and $V_{\theta}^2=-2(y-\frac{2}{3})
  $ in $C^\infty(\mathbb{S}^2\setminus\{S, N\})$ 
  Define 
  $
    V_{\theta}^+ := 2|y-\frac{2}{3}|$ and $V_{\theta}^- := -2|y-\frac{2}{3}|. 
  $ Note $V_{\theta}^{\pm}$ are not smooth at $y_0=2/3$ ($\theta_0=\cos^{-1}(2/3)$). Correspondingly, the streamlines of $v^{\pm}$ are not continuous along $\theta=\theta_0$ 
  (see Figure \ref{fig:ex7_3}). 
  As $\nu \to 0$, the graphs of $U_{\nu, \theta}$ and $V_{\theta}^{1,2}$ is displayed in Figure \ref{fig:ex7_1}. 
\end{example}
\begin{figure}[!htb]
	\centering
    \includegraphics[height=4.5cm]{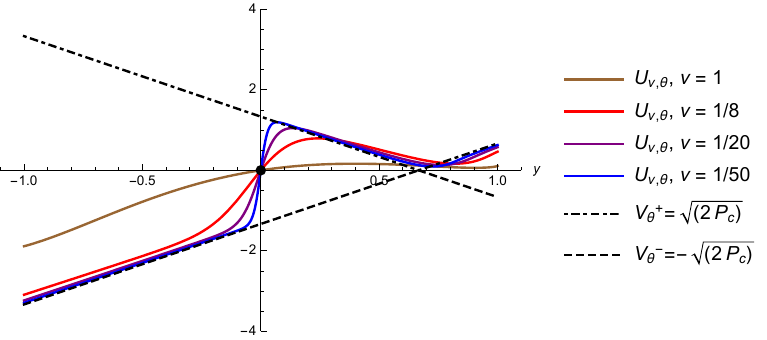}
  	\caption{\small{The graphs of $U_{\nu, \theta}(y)$ and $V_{\theta}^{\pm}$ in Example \ref{exb_1}}}
  	\label{fig:ex7_1}
\end{figure}
In Figure \ref{fig:ex7_1}, we can see that as $\nu\to 0$, $U_{\nu, \theta}(y)\to V_{\theta}^+(y)$ for $y$ in $(0, 1)$, and $U_{\nu, \theta}(y)\to V_{\theta}^-(y)$ for $y$ in $(-1, 0)$. There is a jump discontinuity at $y_1=0$ ($\theta_1=\pi/2$). This means 
$u_{\nu}$ tends to $v^+$ on the upper half sphere, and $u_{\nu}\to v^{-}$ on the lower half sphere (see this phenomenon in Figure \ref{fig:ex7_2} and \ref{fig:ex7_3} of the streamlines of $u_{\nu}$ and $v^{\pm}$).

\begin{figure}[!htb]
	\centering
	(a) \includegraphics[height=4cm]{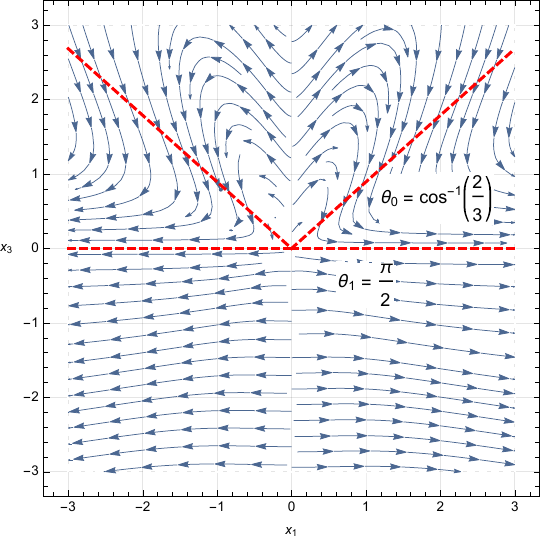} \hspace{1cm} (b) \includegraphics[height=4cm]{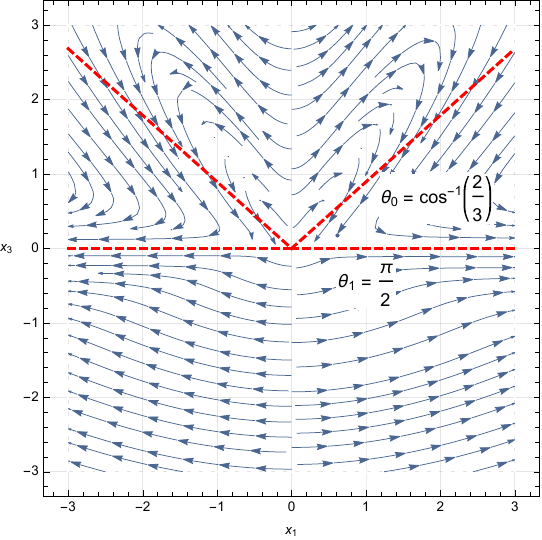}\\
	(c) \includegraphics[height=4cm]{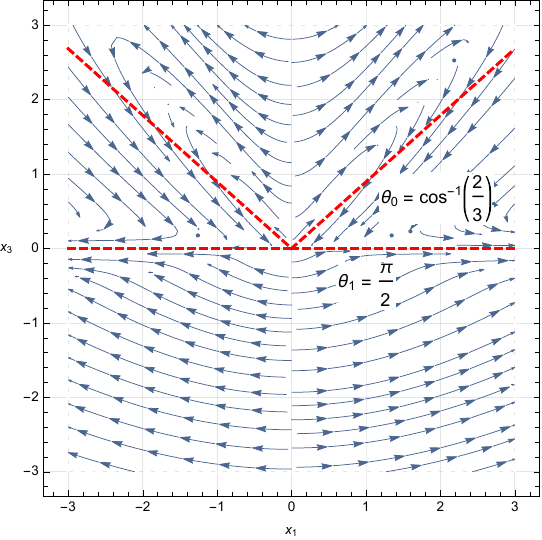} \hspace{1cm} (d) \includegraphics[height=4cm]{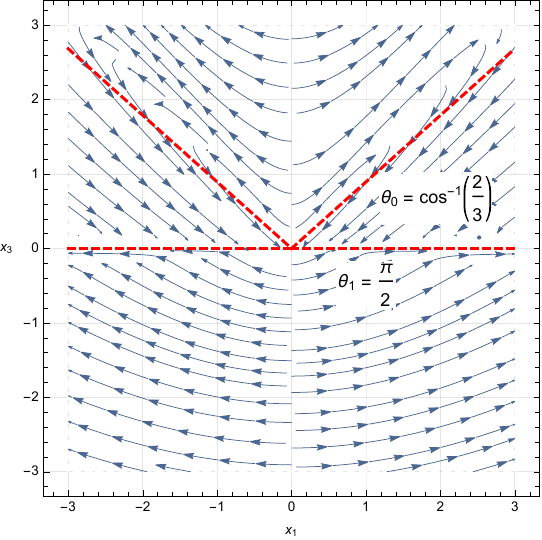}
	\caption{\small{The streamlines of $u_{\nu}$ for (a) $\nu=1$, (b) $\nu=1/8$, (c) $\nu=1/20$, (d) $\nu=1/50$.}} 
	\label{fig:ex7_2}
\end{figure}

\begin{figure}[!htb]
	\centering
	(a) \includegraphics[height=4cm]{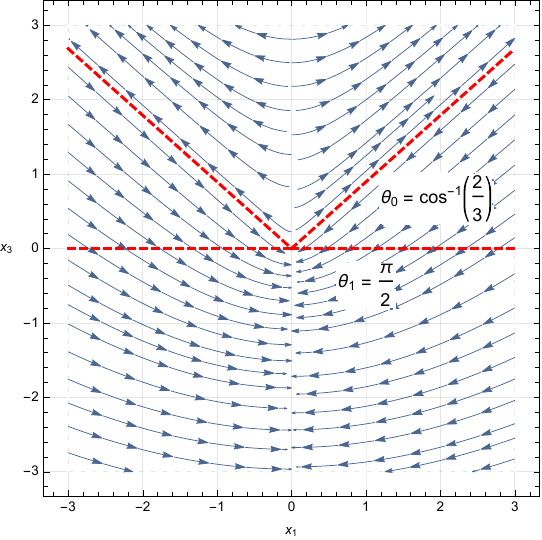} \hspace{1cm} (b) \includegraphics[height=4cm]{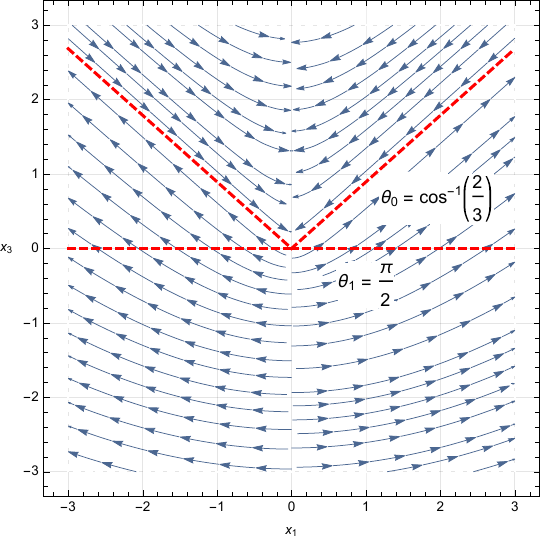}\\
	\caption{\small{(a) The streamlines of $v^{+}$; (b) The streamlines of $v^{-}$.}} 
	\label{fig:ex7_3}
\end{figure}

\FloatBarrier

\bibliographystyle{plain}
\bibliography{refs.bib}

\end{document}